\documentclass
{amsproc}
\usepackage{amsmath}
\usepackage{amsthm}
\usepackage{amssymb}
\usepackage{graphicx}
\usepackage{multirow}

\DeclareGraphicsExtensions{.png,.pdf,.eps}
\usepackage[all,2cell]{xy}
\usepackage{tikz}
\usepackage{pgf}
\usepackage{enumerate}
\usepackage{mathdots}
\usetikzlibrary{cd}
\usepackage{caption}
\usepackage{hyperref}
\usepackage{array}
\newcolumntype{H}{>{\setbox0=\hbox\bgroup}c<{\egroup}@{}}
\usepackage{mathtools} 
\usepackage{tikz-cd} 

\newtheorem{theorem}{Theorem}[section]
\newtheorem{lemma}[theorem]{Lemma}
\newtheorem{corollary}[theorem]{Corollary}
\newtheorem{proposition}[theorem]{Proposition}

\newtheorem*{theorem-non}{Theorem}

\theoremstyle{definition}

\newtheorem{definition}[theorem]{Definition}

\newtheorem{setup}[theorem]{Setup}

\newtheorem{remark}[theorem]{Remark}

\newtheorem{set}[theorem]{Setup}

\title[Toric non-equalized flips associated to $\mathbb{C}^*$-actions]{Toric non-equalized flips associated to $\mathbb{C}^*$-actions}
\author[Barban]{Lorenzo Barban}
\address{Dipartimento di Matematica, Universit\`a degli Studi di Trento, via
Sommarive 14 I-38123 Povo di Trento (TN), Italy}

\email{lorenzo.barban@unitn.it}

\author[Romano]{Eleonora A. Romano}
\address{Dipartimento di Matematica, Universit\`a degli Studi di Genova, via Dodecaneso 35 16146 Genova (GE), Italy}
\email{eleonoraanna.romano@unige.it}

\subjclass[2010]{Primary 14E05, 14M25, 14L30; Secondary 14M17}
\thanks{Second author partially supported by the Grant HA4383/-1 \lq\lq ATAG\rq\rq by the German Science Agency (DFG) and Thematic Einstein Semester \lq\lq Varieties, Polyhedra, Computation\rq\rq by the Berlin Einstein Foundation}

\usepackage[textsize=tiny]{todonotes}

\usepackage{enumitem}

\newcommand\ignore[1]{}

\DeclareMathOperator{\Ext}{Ext}

\newcommand\CC{{\mathbb{C}}}
\newcommand\PP{{\mathbb{P}}}

\newcommand\ZZ{{\mathbb{Z}}}

\def\C{{\mathbb C}}

\def\P{{\mathbb P}}

\def\R{{\mathbb R}}
\def\Z{{\mathbb Z}}

\def\cI{{\mathcal I}}

\def\cN{{\mathcal{N}}}
\def\cO{{\mathcal{O}}}

\def\cQ{{\mathcal{Q}}}

\def\cS{{\mathcal S}}

\def\cU{{\mathcal U}}

\def\cY{{\mathcal Y}}

\def\operatorname#1{\mathop{\rm #1}\nolimits}

\def\Proj{\operatorname{Proj}}

\def\Ext{\operatorname{Ext}}

\def\Hom{\operatorname{Hom}}

\def\Pic{\operatorname{Pic}}
\def\Hom{\operatorname{Hom}}

\def\codim{\operatorname{codim}}

\def\rank{\operatorname{rank}}

\def\deg{\operatorname{deg}}

\newcommand{\pb}{\ar@{}[dr]|{\text{\pigpenfont J}}}
\def\ol{\overline}

\newcommand{\xdasharrow}[2][->]{
\tikz[baseline=-\the\dimexpr\fontdimen22\textfont2\relax]{
\node[anchor=south,font=\scriptsize, inner ysep=1.5pt,outer xsep=2.2pt](x){#2};
\draw[shorten <=3.4pt,shorten >=3.4pt,dashed,#1](x.south west)--(x.south east);
}}

\newcommand\lra{\longrightarrow}

\newcommand{\git}{\mathbin{
		\mathchoice{/\mkern-6mu/}
		{/\mkern-6mu/}
		{/\mkern-5mu/}
		{/\mkern-5mu/}}}

\begin{document}
\begin{abstract} Starting from $\C^*$-actions on complex projective varieties, we construct and investigate birational maps among the corresponding extremal fixed point components. We study the case in which such birational maps are locally described by toric flips, either of Atiyah type or so called \lq\lq non-equalized\rq\rq. We relate this notion of toric flip with the property of the action being non-equalized. Moreover, we find explicit examples of rational homogeneous varieties admitting a $\C^*$-action whose weighted blow-up at the extremal fixed point components gives a birational map among two projective varieties that is locally a toric non-equalized flip. 
\end{abstract}
\maketitle
\tableofcontents
	

\section{Introduction}\label{sec:intro}

The idea of investigating the relationship between birational geometry and torus actions finds its roots in the works of Reid, Thaddeus, W\l odarczyk (see for instance \cite{ReidFlip, Thaddeus1996, Wlodarczyk}). This connection is often enlightened by the study of the local toric geometry involved: W\l odarczyk has indeed used it in order to prove the Weak Factorization conjecture by means of the construction of the \emph{cobordism}: this notion was firstly studied in the toric case by Morelli (see \cite{Morelli}), and allows to understand birational maps of algebraic varieties in terms of quotients of varieties admitting $\C^*$-actions. In the last years, this research has been carried out by, among the others, Occhetta, Romano, Sol\'a Conde, Wi\'sniewski (cf. \cite{RW, WORS1, WORS2, WORS3}). 

This paper takes inspiration in the work of the previous authors and aims to study some results in a broader context. In particular, focusing on complex projective varieties admitting $\C^*$-actions, the main purpose will be to investigate special birational maps arising among some fixed point components, called extremal. We will deal with the case in which such birational maps are locally given by toric flips, which we explicitly describe and call \emph{toric non-equalized flips}.  

\subsection{Overview}
We will consider nontrivial $\C^*$-actions (with coordinate $t\in \C^*$) on a complex smooth projective variety $X$ of arbitrary dimension. Among the irreducible fixed point components, there exist unique smooth varieties $Y_-$ and $Y_+$, called \emph{sink} and \emph{source}, containing the limit of the generic orbit for $t\to \infty$ and $t\to 0$, respectively. All the other fixed point components are called \emph{inner components}. 

Given an ample line bundle $L$ on $X$, and taking a linearization of it, we may associate an integer to every irreducible fixed point component. By ordering increasingly all the possible values occurring, without taking care of the multiplicity, we will obtain a chain of integers $a_0<\ldots< a_r$. The integer $r$ is called  \emph{criticality} of the $\C^*$-action on $X$. Moreover, there is another invariant of the $\C^*$-action on the pair $(X,L)$, called \emph{bandwidth}, obtained as the absolute value of the difference among the weights of the sink and source. 

The birational geometry of varieties admitting $\C^*$-action of small bandwdith has been investigated in \cite{WORS1}, under the assumption that the action is \emph{equalized} (cf. Definition \ref{def:equalized} and Remark \ref{rem:eq}). For instance, in the equalized case of bandwidth 2, after blowing up $X$ along the sink and the source, there exists an \emph{Atiyah flip} among the corresponding exceptional divisors; we refer to \cite[$\S$7]{WORS1} for details. 

In this paper, we will mainly focus on $\C^*$-actions of criticality $2$, and we start extending the previous methods to the case of \emph{non-equalized} actions.
Therefore, this article lays the groundwork to further investigation regarding non-equalized $\C^*$-actions and the birational geometry of the involved varieties. In particular, following the spirit of \cite{WORS1}, we aim to investigate birational maps arising from $\C^*$-varieties, obtained as follows: we consider the weighted blow-up along the sink and the source of the action on $X$, to get a birational map among the corresponding exceptional divisors $\psi: Y^\flat_-\dashrightarrow Y^\flat_+$ (see Lemma \ref{lem:birational_map}). In this setting, we formulate our main result as follows, where we refer to Definition \ref{def:bordism} for the notion of \emph{bordism}.

\begin{theorem-non} \label{thm}
Let $X$ be a smooth projective variety with a faithful $\C^*$-action of criticality $2$. Let $X^\flat$ be the weighted blow-up of $X$ along the sink $Y_-$ and the source $Y_+$, and assume that $X^\flat$ is a bordism. Then $\psi: Y^\flat_-\dashrightarrow Y^\flat_+$ is locally a toric Atiyah flip if and only if the $\C^*$-action on $X^\flat$ is equalized at every inner component. 
\end{theorem-non} 

The purpose of this paper is twofold: on one side we prove the above result, which shows that the local behaviour of the flip can be easily detected by studying if the $\C^*$-action is equalized at every inner component. On the other side, we find explicit examples of birational maps among projective varieties, corresponding to the sink and source of $\C^*$-varieties. In this context, we provide examples of $\C^*$-actions on \emph{smooth projective} varieties whose associated birational map $\psi$ is locally a toric flip, on one hand of Atiyah type (see subsection \ref{Atiyah_example}), and on the other not of Atiyah type, which we will call \emph{toric non-equalized flip} (see subsection \ref{example_noneq}).

\subsection{Contents}

This paper is organized as follows. Section 2 introduces the necessary background regarding $\C^*$-actions on complex smooth projective varieties. 

Section 3 represents the technical core of the paper; it focuses on the interplay between birational geometry and torus actions. We start by recalling the existence of a birational map among the exceptional divisors associated to the sink and the source (see Lemma \ref{lem:birational_map}), and construct the necessary background to deal with the main results of the next section. To this end, we first generalize to the non-equalized case the local toric construction given in \cite[$\S$5]{WORS1}, by introducing the notion of toric non-equalized flip, and we also prove that given two birationally equivalent toric varieties $X_-$ and $X_+$, there exists a new toric variety endowed with a $\C^*$-action, such that the sink and the source are respectively $X_-$ and $X_+$ (see Proposition \ref{toricbordism}). We conclude this section with the description of the weighted blow-up of a $\C^*$-variety along an extremal component, discussing how this generalizes the well-known toric weighted blow-up (see Remark \ref{rem:toric_wb}). 

In Section 4 we show the main Theorem (see also Theorem \ref{thm:main}), and we deduce Corollary \ref{cor:main} which will be used to study the examples in the next section. The proof of these results relies on the local toric construction given in the previous section. We conclude this part by generalizing the main result to the case of arbitrary criticality (see Corollary \ref{cor:maincrit}).

Section 5 is devoted to construct two different classes of examples illustrating the phenomena described above. In the first case, we study a $\C^*$-action on a smooth quadric hypersurface: this is an example of equalized action at the inner component of criticality and bandwidth 2, such that the birational map among the exceptional divisors is locally given by a toric Atiyah flip (see subsection \ref{Atiyah_example}, Proposition \ref{prop:ex1}).
The second example is given by a $\C^*$-action on the Grassmannian of lines of a smooth quadric hypersurface: this action has criticality 2, bandwidth 4, and is not equalized at the inner component, so that the corresponding birational map is locally given by a toric non-equalized flip (see subsection \ref{example_noneq}, Proposition \ref{prop:ex2}).

\noindent\medskip\par{\bf Acknowledgements.} We would like to dedicate this work to Bia{\l}ynicki-Birula: his theory has been our main source of inspiration. We are grateful to Luis Sol\'a Conde for all the stimulating conversations and important suggestions. We would like also to thank Alberto Franceschini for helpful discussions regarding torus actions, and Gianluca Occhetta and Jarek Wi{\'s}niewski for all the ideas coming from previous collaborations. The initial ideas of this work arose during a stay of the second author at Freie Universit\"at Berlin in occasion of the semester \lq \lq Varieties, Polyhedra, Computation\rq \rq. She would like  to thank Klaus Altmann and Karin Schaller for fruitful discussions on toric geometry, and all the organizers for the kind hospitality. 

\section{Preliminaries}\label{sec:prelim}

This section is devoted to recall and collect some background material on $\C^*$-actions. We subdivide the preliminaries into three subsections where corresponding references for further details are given. 

\subsection{Bia{\l}ynicki-Birula decomposition} \label{sub:BB} Let $X$ be a complex smooth projective variety with a nontrivial $\C^*$-action. We denote by $X^{\C^*}$ the fixed locus of such an action, and by $\cY$ the set of the irreducible fixed point components, so that we may write 
$$X^{\C^*}=\bigsqcup_{Y\in \cY}Y.$$ It is known (see \cite{SOM}) that for $x\in X$ the action $\C^*\times
\{x\}\to X$ extends to a holomorphic map $\PP^1\times \{x\}\to X$, hence there
exist $\lim_{t\rightarrow 0}t x$, and $\lim_{t\rightarrow \infty}tx$.
Moreover, since the orbits are locally closed and the closure of an orbit is an
invariant subset, then both the limit points of an orbit lie in $\cY$. We will
call these limits the {\em source} and the {\em sink} of the orbit of $x$,
respectively.
Given $Y\in \cY$, we denote by 
\begin{equation*}\label{eq:BBcells}
	X^+(Y):=\{x\in X\mid \lim_{t\to 0} tx\in Y\}, \hspace{0.3cm}
	X^-(Y):=\{x\in X\mid \lim_{t\to \infty} tx\in Y\}
\end{equation*}
the {\it Bia{\l}ynicki-Birula cells} of the action. 

Moreover, since $X$ is smooth it is well known that every fixed point component $Y\in\cY$ is smooth (cf. \cite[Main theorem]{IVERSEN}), and the normal bundle of $Y$ in $X$ splits into two subbundles on which $\C^*$ acts with positive and negative weights, respectively:
\begin{equation*} \label{equation_normal}
\cN_{Y|X}\simeq \cN^+(Y)\oplus \cN^-(Y).
\end{equation*}

We will set $\nu^\pm(Y):=\rank \cN^\pm(Y)$.

The following result is due to Bia{\l}ynicki-Birula and we will refer to it as the {\em BB-decomposition}. We state it as presented in \cite[Theorem 4.2]{CARRELL}. See
\cite{BB} for the original exposition.
\begin{theorem}
\label{thm:BB_decomposition}
In the situation described above, for every $Y\in \cY$, the following hold:
 \begin{itemize}

\item [(1)] $X^{\pm}(Y)$ are locally closed subsets and there are two decompositions
 $$X=\bigsqcup_{Y\in \cY}X^+(Y)=\bigsqcup_{Y\in \cY}X^-(Y).$$
\item [(2)] There are $\C^*$-equivariant isomorphisms $X^{\pm}(Y)\simeq \cN^{\pm}(Y)$ lifting the natural maps
 $X^\pm(Y)\rightarrow Y$. Moreover, the map $X^\pm(Y)\rightarrow Y$ is algebraic and is a $\CC^{\nu^\pm(Y)}$-fibration.
\end{itemize}
\end{theorem} 

\begin{remark}\label{rem:sinksource}
As a consequence of Theorem \ref{thm:BB_decomposition} (1), there exists a unique component $Y_+\in \cY$ (resp. $Y_-\in \cY$) such that $X^+(Y_+)$ (resp. $X^-(Y_-)$) is a dense open set in $X$. We will call $Y_+$ and $Y_-$ the \textit{source} and the \textit{sink} of the action, and refer to them as the \textit{extremal fixed components} of the action. We will call  \textit{inner components}  the fixed point components which are not extremal.
\end{remark}

\subsection{Linearizations of $\C^*$-actions}  Given a nontrivial $\C^*$-action on a complex smooth projective variety $X$ as in the previous section, we will also consider a linearization $\mu_L$ of the $\C^*$-action on an ample line bundle $L\in\Pic(X)$; note that linearizations exist for every line bundle on a normal projective variety (cf. \cite{KKLV}). We refer the reader to \cite[$\S$2.3]{RW} and references therein for details on linearized line bundles. We then consider the action of $\C^*$ on the pair $(X,L)$. 
Moreover, we will choose an isomorphism of the lattice of characters of $\C^*$ with $\Z$, so that to every fixed point component $Y\in \cY$ is assigned a weight $\mu_L(Y)\in\Z$. 
Let us consider these weights $\mu_L(Y)$, for $Y\in \cY$, and rearrange them in an increasing order, taking care to consider them without multiplicity (so that if there are two or more components $Y,Y'\in \cY$ such that $\mu_L(Y)=\mu_L(Y')=a$, we consider the weight $a$ only one time): we will obtain a chain of the form
$$a_0< a_1 < \ldots < a_r.$$
We define the \emph{criticality} of the $\C^*$-action on a pair $(X, L)$ as the integer $r$. This notion has been introduced in \cite{WORS3}; we refer to that paper for recent results on $\C^*$-actions on pairs $(X,L)$, depending on their criticality.  

Using that $L$ is ample, it easily follows that the minimal and the maximal value of $\mu_L$ are achieved at the sink and the source of the action. The {\em bandwidth} of the $\C^*$-action on $(X,L)$ is defined as the difference between the maximal and the minimal value of $\mu_L$. Note that two linearizations of $L$ differ by a character, hence the above definitions are well-defined. We stress that, while the concept of bandwidth and criticality may look similar, one may easily find examples of $\C^*$-actions on polarized pairs where the criticality and the bandwidth do not coincide (see for example $\S$5.2, in particular Lemma \ref{lem:prelim_ex2} and the subsequent Remark).

We conclude this section by recalling the AM vs FM equality, which has been introduced in \cite[$\S$2.A]{RW}, and relates the amplitude of a line bundle on $\PP^1$ with the weights of the action on the fibers of the line bundle over the fixed points. 

\begin{lemma}\label{lem:AMvsFM} \cite[Lemma 2.2]{RW}
Let $\C^*\times\PP^1\rightarrow\PP^1$ be an
action with source and sink $y_+$ and $y_-$. Consider a line bundle $L$
over $\PP^1$ with linearization $\mu_L$. Then
\begin{equation}\tag{AM vs FM}
\mu_{L}(y_+)-\mu_{L}(y_-)=\delta(y_+)\deg{L}
\end{equation}
where $\delta(y_+)$ is the weight of the action on the tangent space $T_{y_+}\PP^1$.
\end{lemma}

\subsection{Equalized, non-equalized $\C^*$-actions, and bordisms} Let us keep the notation introduced in subsection \ref{sub:BB}. We recall some notions from \cite{RW, WORS1}.


\begin{definition} \label{def:equalized}
We say that the action of $\C^*$ on $X$ is {\em  equalized} at $Y\in\cY$ if the weights of the action on $\cN^{\pm}(Y)$ are all equal to $\pm 1$. Otherwise we say that the action is {\em  non-equalized} at $Y$.

If the $\C^*$-action is equalized at every fixed point component, we will say that the action is equalized.
\end{definition}
\begin{remark} \label{rem:eq} Notice that Theorem \ref{thm:BB_decomposition} (2) gives a geometric insight on the notion of equalization. Indeed, given $Y\in \cY$ the isomorphism $X^{\pm}(Y)\simeq \cN^{\pm}(Y)$ sends one-dimensional orbits to curves in the fibers of $\cN^{\pm}(Y)\to Y$. Take a point $y\in Y$, and let $C$ be the closure of an orbit having sink (or source) in $y$. Consider its normalization $f\colon \P^1\to C$; then $\C^*$ acts on $\P^1$ with sink $z_-$ (and source $z_+$) such that $f(z_-)=y$ (respectively, $f(z_+)=y$). We may then identify the weights of the action on $\cN^{-}(Y)$ (resp. of $\cN^{+}(Y)$) with the weights of the induced $\C^*$-action on $T_{\P^1,z_{-}}$ (resp. $T_{\P^1,z_{+}}$). Therefore, the images of the closures of one-dimensional orbits into $\cN^{\pm}(Y)$ are all lines if and only if the $\C^*$-action is equalized. This also says that the $\C^*$-action is equalized if and only if for every $Y\in \cY$ and $x\in (X^{-}(Y)\cup X^{+}(Y))\setminus Y$ the isotropy group of the $\C^*$-action on $x$ is trivial (see also \cite[Corollary 2.9]{WORS1}).
\end{remark}

\begin{definition}\label{def:bordism}
A {\em B-type action} is a $\C^*$-action on $X$ whose extremal fixed point components have codimension one. A {\em bordism} is a B-type action such that  $\nu^\pm(Y)\geq 2$ for every inner component $Y$. 
\end{definition}



\section{Birational Geometry arising from $\mathbb{C}^*$-actions}\label{sec:birgeo}
This section is devoted to explain the technical tools which we will need to prove the main results, involving toric flips arising from $\C^*$-actions.  
The first step in this direction is the construction introduced in \cite{WORS1} of a special birational map between the source and the sink of a B-type $\C^*$-action on a smooth projective variety $X$. 
Let $Y_+$ and $Y_-$ be the source and the sink of such an action, respectively. For any $Y\in \cY$, let us consider the subvarieties $Z_-(Y):=\ol{X^+(Y)}\cap Y_{-}\subset Y_{-}$, $Z_+(Y):=\ol{X^-(Y)}\cap Y_{+}\subset Y_{+}$ and set: 
\begin{equation} \label{equation_Z}
	Z_-:=\bigcup_{Y\in \cY}Z_-(Y),\quad  Z_+:=\bigcup_{Y\in \cY}Z_+(Y).
\end{equation}

\begin{lemma} \cite[Lemma 3.4]{WORS1} \label{lem:birational_map} \label{lem:birational_map_b} Let $X$ be a complex smooth projective variety with a B-type $\C^*$-action. 
	Then there exists an isomorphism:
	$$
	\psi:Y_{-}\setminus Z_{-}\lra Y_{+}\setminus Z_+,
	$$
	assigning to every point $x$ of $Y_{-}\setminus Z_{-}$ the limit for $t \to 0$ of the unique orbit having limit for $t \to \infty$ equal to $x$.
\end{lemma}
When the $\C^*$-action on $X$ is a bordism, applying \cite[Lemma 3.9]{WORS1} we know that the birational map $\psi\colon Y_-\dashrightarrow Y_+$ introduced in the above Lemma is an isomorphism in codimension one. On the other hand, under the assumption that the $\C^*$-action on $X$ is equalized at $Y_{-}$ and $Y_{+}$, it has been studied when the blow-up at the sink and at the source gives rise to a variety whose extremal fixed components are isomorphic in codimension one (see \cite[Lemma 3.10]{WORS1}). We will investigate what happens in the more general case of a non-equalized $\C^*$-action. To this end, we will first present toric flips that we call \emph{non-equalized}, and discuss the toric bordism construction in subsection \ref{sec:toric_bordism}. Finally, we will introduce some preliminaries on weighted blow-ups in subsection \ref{sec:weighted}. 

\subsection{Toric bordism and non-equalized flips} \label{sec:toric_bordism} In this subsection we introduce a  generalization of the notion of \emph{toric bordism} presented in \cite[Definition 5.8]{WORS1}. We will use this construction to describe locally the birational map introduced in Lemma \ref{lem:birational_map}. We refer to \cite{CLS} for the proofs of the results we will use. See \cite{ReidToric, JW-Toric} for a detailed study on toric flips, \cite{MFK} for details on GIT theory and the definition of good and geometric quotients, and \cite{Wlodarczyk, Morelli} for the Morelli-W\l odarczyk cobordism, which we recall in the following:

\begin{definition}\label{def: cobordism}
	Let $X_1,X_2$ be birationally equivalent normal varieties. The \emph{Morelli--W\l odarczyk cobordism} between $X_1$ and $X_2$ is a normal variety $B$, endowed with a $\C^*$-action such that 
	\begin{equation*}
		\begin{split}
			B_-&=\{p\in B\mid \lim_{t\to 0} tp \text{ does not exists}\},\\
			B_+&=\{p\in B\mid \lim_{t\to \infty} tp \text{ does not exists}\}
		\end{split}
	\end{equation*}
	are non-empty open subsets of $B$, such that there exist geometric quotients $B_+/\C^*$, $B_-/\C^*$ satisfying
	$$B_+/\C^*\simeq X_1 \dashrightarrow X_2\simeq B_-/\C^*,$$
	where the birational equivalence is realized by the open subset $$\mathcal{V}=(B_+\cap B_-)/\C^*\subset B_{\pm}/\C^*.$$
\end{definition}

We now recall the toric version of the Morelli-W\l odarczyk cobordism.

\begin{setup}\label{setupcobordism}
	Consider a lattice $N$ of rank $n+1$, generated by the canonical basis of $\R^{n+1}$ which we will denote by $e_1,\ldots,e_{n+1}$, and consider $d_1, d_2\in \mathbb{N}$ such that $1 < d_1 \leq d_2 < n+1$.
	Consider the simplicial cone $\delta= \langle e_1,\ldots,e_{n+1}\rangle,$ which determines the affine toric variety $X_\delta=\C^{n+1}$. 
	
	Let $q_1,\ldots,q_{d_1},q_{d_2+1},\ldots,q_{n+1}$ be positive integers, which, up to dividing for their greatest common divisor, we may assume to be coprime, and consider the faithful action of $\C^*$ on $X_\delta$ given by, for $t\in\C^*$ and $p=(p_1,\ldots,p_{n+1})\in X_\delta$, 
	$$\C^*\times\C^{n+1}\to \C^{n+1}$$
	$$(t,p)\mapsto (t^{-q_1}p_1,\ldots,t^{-q_{d_1}}p_{d_1}, p_{d_1+1},\ldots,p_{d_2}, t^{q_{d_2+1}}p_{d_2+1},\ldots,t^{q_{n+1}}p_{n+1}).$$
	This action has an associated $1$-parameter subgroup $v\in N$ of the form
	$$v=(-q_1,\ldots,-q_{d_1},0,\ldots,0,q_{d_2+1},\ldots,q_{n+1}).$$
\end{setup} 
For simplicity, we will denote the coordinate of a point $p\in \C^{n+1}$ by $$(p_1,\ldots,p_{n+1})=(p_-,p_0,p_+)$$ in order to stress on which coordinates $\C^*$ acts with negative (resp. zero, positive) weights.

Let us also consider the quotient lattice $\overline{N}=N/\Z v$, and denote by $\pi: N\to \overline{N}$ the natural projection; we will continue to denote by $e_i$ their image under $\pi$. 

In order to make this discussion more clear, we will develop an intuitive picture-like description of this construction: for now, we describe $X_{\delta}$ as follows:

\begin{center}
	\includegraphics[scale=0.08]{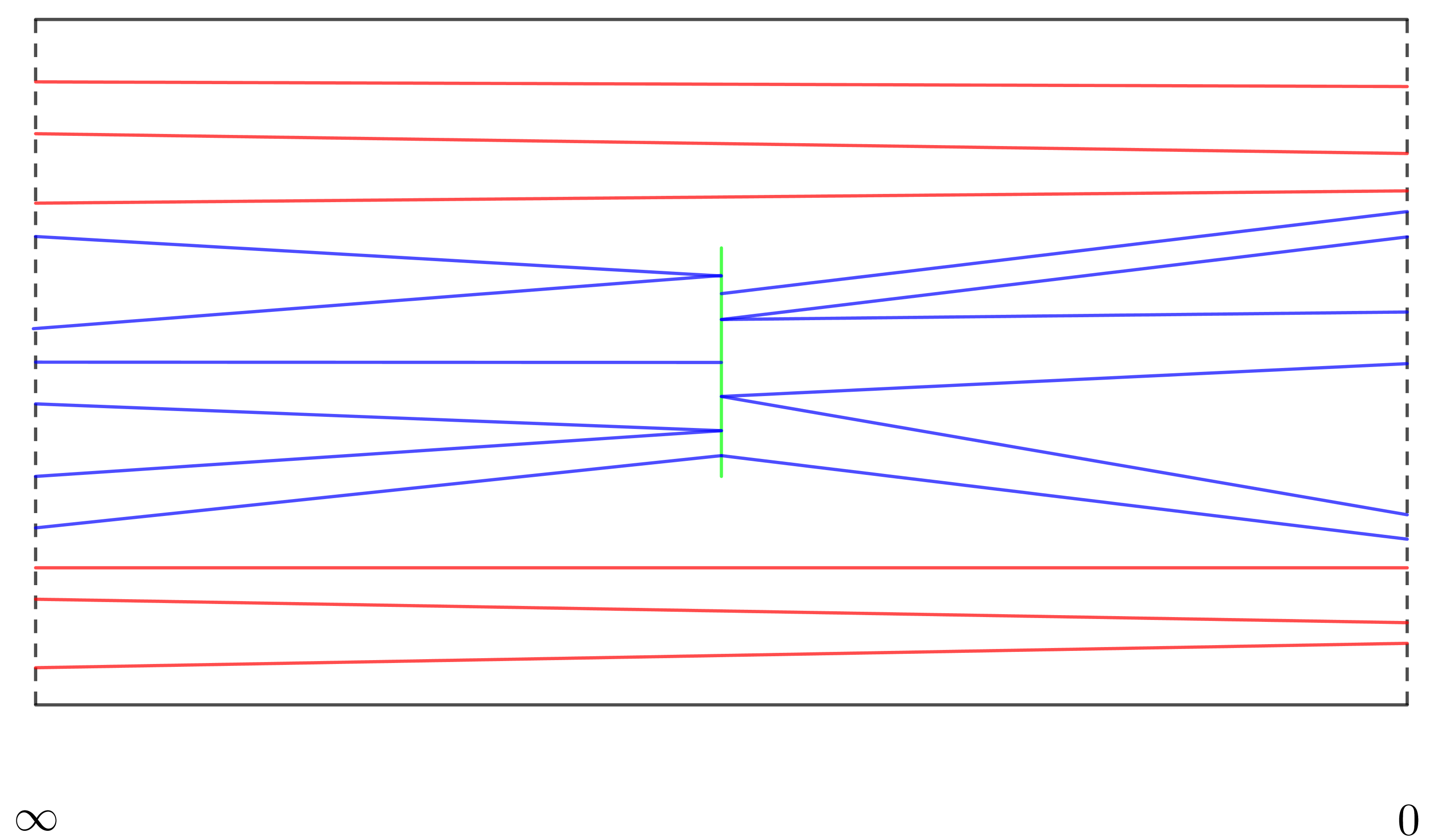}
\end{center}
The central component represents the fixed point component $\C^{d_2-d_1}$, the horizontal lines are the orbits of the points, and the left and right side correspond to the limit of the orbits for $t\to \infty$ (respectively $t\to 0$), for $t\in\C^*$.

\begin{remark}{\cite[Lemma p. 264]{JW-Toric}}\label{remark:affinegit}
	The toric variety whose associated cone is the image $\overline{\delta}=\pi(\delta)\subset \overline{N}$, is the good quotient $\C^{n+1}\git \C^*$ of $X_{\delta}$ under the $\C^*$-action of Setup \ref{setupcobordism}.
\end{remark}

\begin{proposition} \label{prop:description_B}
	Let us consider the $\C^{*}$-action on $\C^{n+1}$ described in Setup \ref{setupcobordism}. Then the open subsets $B_+,B_-$ of Definition \ref{def: cobordism} correspond to 
	\begin{equation*}
		\begin{split}
			B_+&=\{p=(p_-,p_0,p_+)\in \C^{n+1}\mid p_+\neq (0,\ldots,0)\}=\C^{n+1}\setminus \{(p_-,p_0,0)\}, \\
			B_-&=\{p=(p_-,p_0,p_+)\in \C^{n+1}\mid p_-\neq (0,\ldots,0)\}=\C^{n+1}\setminus \{(0,p_0,p_+)\}.\\
		\end{split}
	\end{equation*}
	Moreover $B_+,B_-$ are toric varieties, whose associated fans in $N$, which we will denote by $\Delta_{\pm}$, can be described as follows:
	$$\Delta_+=\{\tau\preccurlyeq \Sigma(\delta)\mid \tau \not\supset \langle e_i\rangle \text{ for } i=d_2+1,\ldots,n+1\},$$
	$$\Delta_-=\{\tau\preccurlyeq \Sigma(\delta)\mid \tau \not\supset \langle e_i \rangle \text{ for } i=1,\ldots,d_1\},$$
	where by $\Sigma(\delta)$ is the natural fan associated to $\delta$.
\end{proposition}
\begin{proof}
	Let us prove only the case of $B_+$, being the other similar. Considering the $\C^{*}$-action on $\C^{n+1}$ described in Setup \ref{setupcobordism}, the identity $B_+=\{p\in X_{\delta}\mid p_+\neq 0\}=\C^{n+1}\setminus \{(p_-,p_0,0)\}$ follows by definition. By the orbit-cone correspondence (see \cite[Theorem 3.2.6]{CLS}) the set given by all the points of $\C^{n+1}$ of the form $\{(p_-,p_0,0)\}$ is the union of orbits associated to the cones of $\delta$ which do not contain $\langle e_{d_2+1}, \dots ,e_{n+1} \rangle$, therefore the claim. 
 \end{proof}
We then may visualize $B_-,B_+$ as respectively:

\begin{center}
	\includegraphics[scale=0.08]{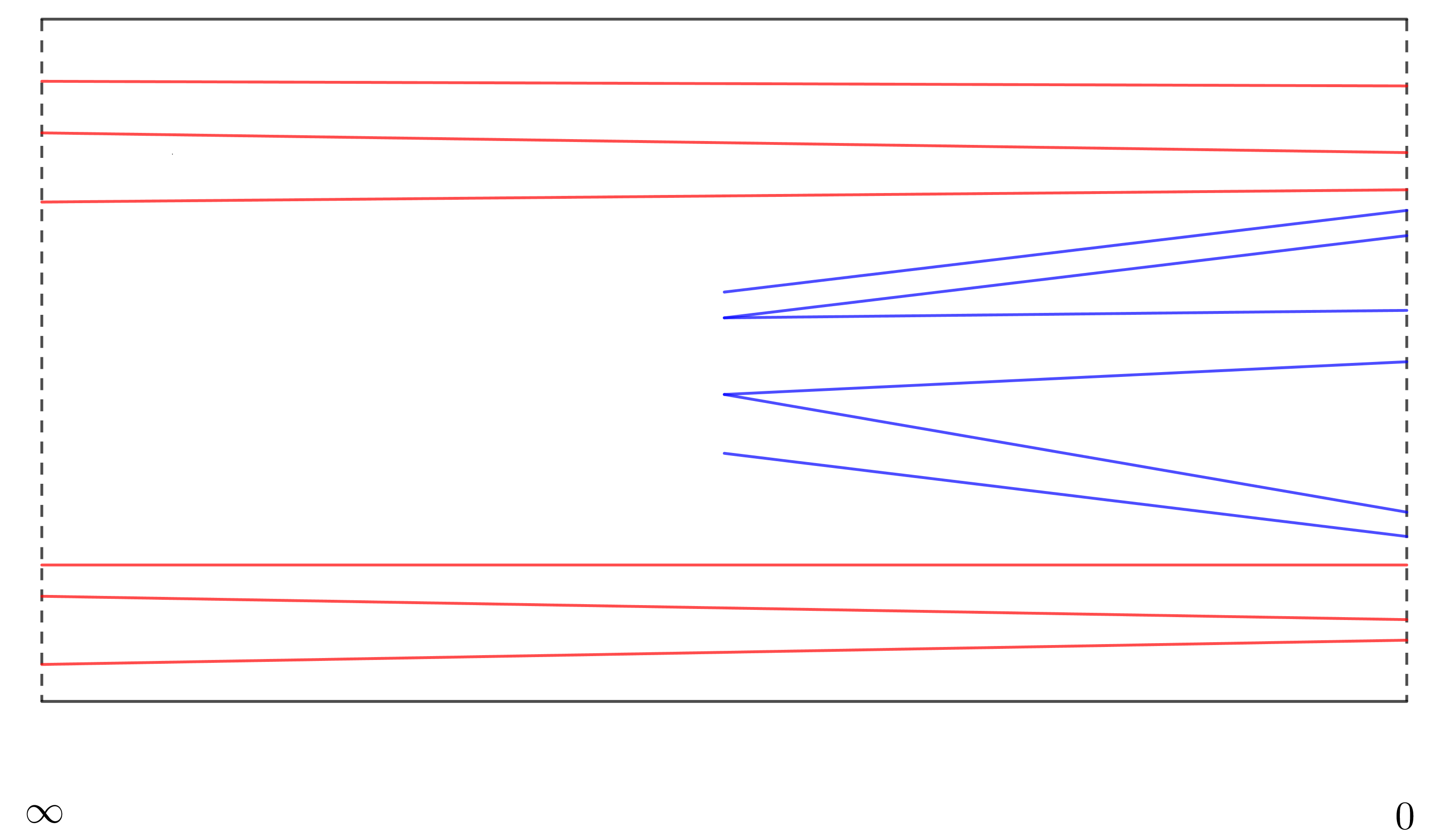}\hspace{0.5cm}\includegraphics[scale=0.08]{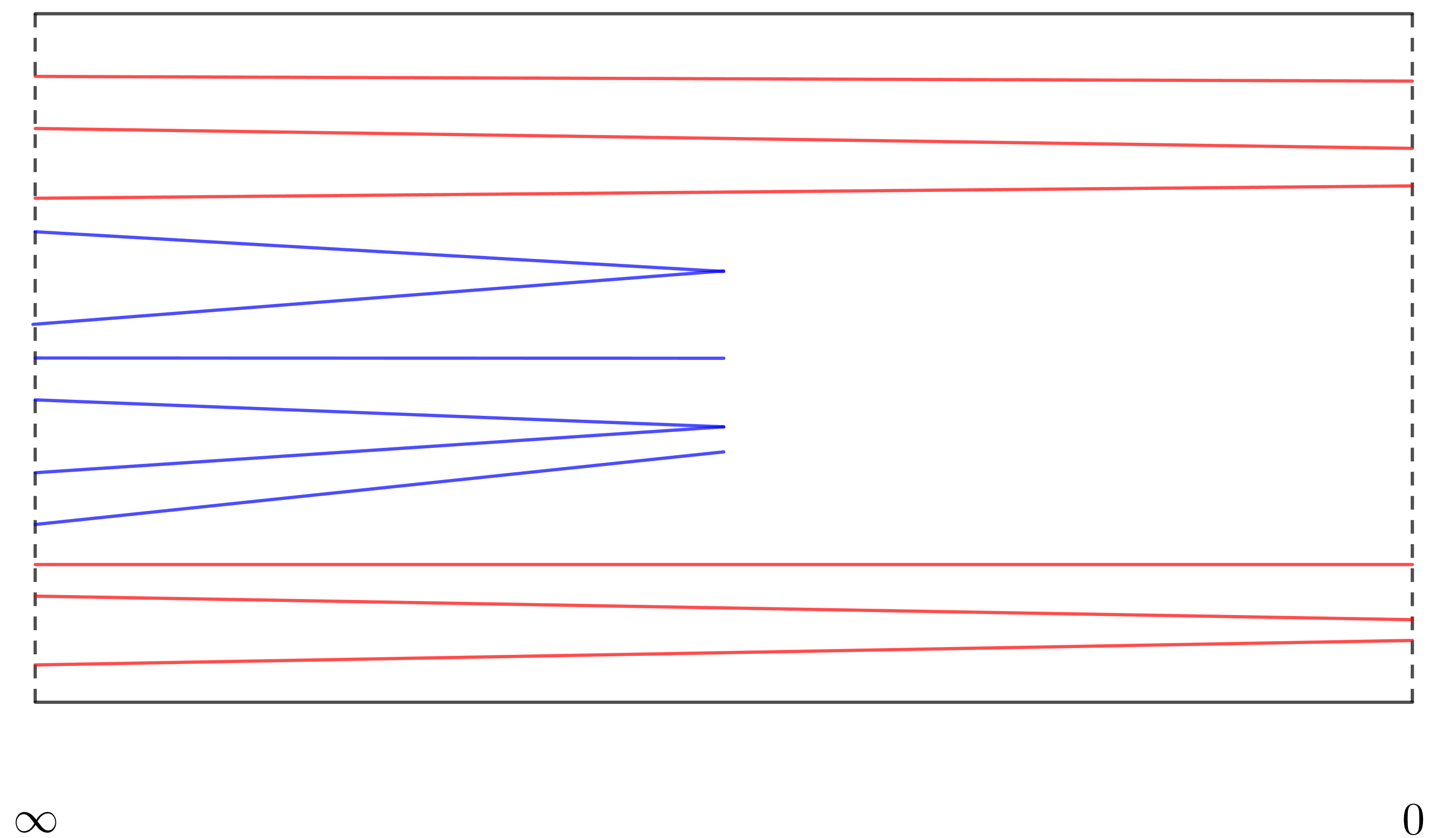}
\end{center}

Indeed we are removing the points whose limit, for $t\in\C^*$ going to $0,\infty$, does not converge.

\begin{remark}
	In the 3-dimensional case the maximal cones of $\Delta_-$ (respectively $\Delta_+$) can be easily  detected by looking at the maximal cones visible from $v$ (respectively $-v$). This is also consistent with the description given in \cite[Example 2]{Wlodarczyk} and it is an easy application of the orbit-cone correspondence.
	
Consider for example the cone $\delta=\langle e_1,e_2,e_3 \rangle$ and the action of $\C^*$ on $\C^3$ given by $v=(-2,-1, 1)$. Then the corresponding cones of maximal dimension are: $$\Delta_+(2)=\{ \langle e_1,e_2\rangle\}, \hspace{0.3cm} \Delta_-(2)=\{ \langle e_1,e_3 \rangle , \langle e_2,e_3\rangle \}.$$
	\begin{center}
		\includegraphics[scale=0.15]{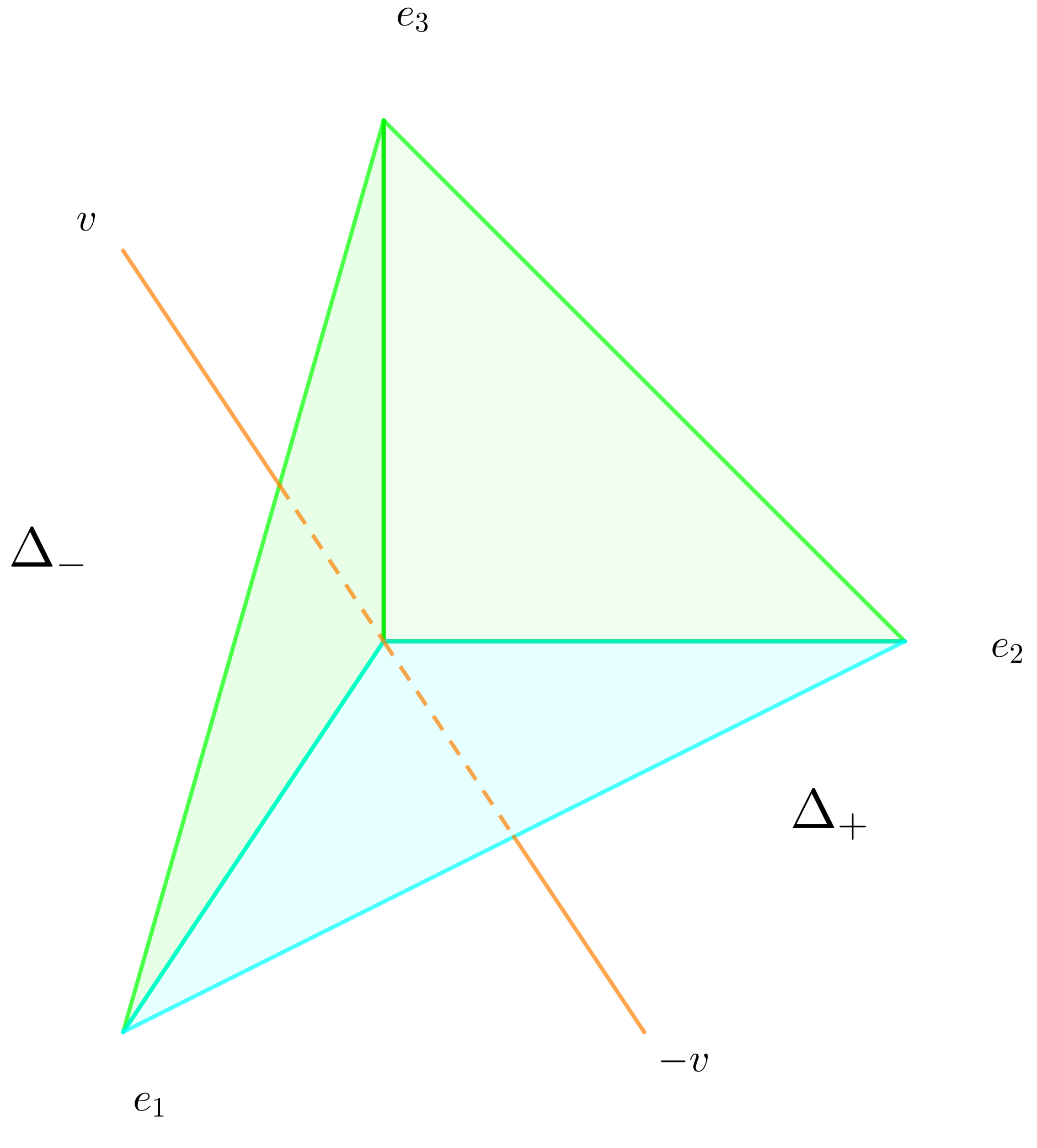}
	\end{center}
\end{remark}

Taking into account Proposition \ref{prop:description_B}, we set $\overline{\Delta_{\pm}}:=\pi(\Delta_{\pm})$, and from now on we denote the associated toric varieties by $X_{\mp}=X_{\overline{\Delta}_{\pm}}$. The reason behind this apparent misleading notation is justified by the fact that $X_-$ and $X_+$, as we will see, are respectively the sink and the source of a $\C^*$-action.

\begin{remark} \label{rem:cobordism}
The morphisms $B_{\mp}\to X_{\pm}$ are geometric quotients; in particular they are $\C^*$-bundles (cf. \cite[Example 2]{Wlodarczyk},\cite[Lemma p. 265]{JW-Toric}). Therefore, we deduce that $\C^{n+1}$ gives a cobordism between the varieties $X_-$ and $X_+$. 
\end{remark}

\begin{lemma} \label{lem:flip}
	There exists a toric flip among $X_-$ and $X_+$.
\end{lemma}

\begin{proof}
	As noted in \cite[$\S$5.2, p. 21]{WORS1} and in \cite[p. 265]{JW-Toric}, the fans $\overline{\Delta_{\pm}}$ determine two simplicial subdivisions of $\overline{\delta}$, that is
	$$\overline{\delta}=\bigcup_{i=1}^{d_1} \overline{\delta_i}=\bigcup_{i=d_2+1}^{n+1}\overline{\delta_i},$$
	where by $\overline{\delta_i}$ we mean the image under $\pi$ of the cone $\delta_i=\langle e_1,\ldots, \hat{e}_{i}, \ldots, e_{n+1}\rangle$. It is well know that the map associated with the operation of replacing one subdivision with the other is a flip (see for istance \cite[Theorem 3.4]{ReidToric} or \cite[$\S$3]{JW-Toric}), hence the claim. 
\end{proof}

In the situation described in Set-up \ref{setupcobordism}, since in view of Lemma \ref{lem:flip} there exists a toric flip $X_-\dashrightarrow X_+$, it makes sense to introduce the notion of \emph{toric Atiyah flip} and \emph{toric non-equalized flip} as follows. 

\begin{definition}\label{toricatiyahflip}
	If the non-zero weights of the $\C^*$-action on $\C^{n+1}$ are equal to $\pm 1$, then the birational transformation $X_- \dashrightarrow X_+$ is called \emph{toric Atiyah flip}. Otherwise it will be called \emph{toric non-equalized flip}. 
\end{definition}

The clean distance between this terminology lies  in the fact that, while the former flip is well-known in the literature, the latter one has a deep connection on the property of being the $\C^*$-action non-equalized, as we will see in the next section.

At this point, the Morelli-W\l odarczyk cobordism (see Remark \ref{rem:cobordism}) may be represented by means of the following diagram:

\begin{center}
	\begin{tikzcd}
		B_+ \arrow[r, hook] \arrow[d]                            & \C^{n+1} \arrow[dd]   & B_- \arrow[l, hook'] \arrow[d]        \\
		X_-= B_+/\mathbb{C}^* \arrow[rr, dashed] \arrow[rd] &                & X_+= B_-/\mathbb{C}^* \arrow[ld] \\
		& \C^{n+1} \git \mathbb{C}^* &                                      
	\end{tikzcd}
\end{center}
where the natural toric morphism $\C^{n+1}\longrightarrow \C^{n+1} \git \mathbb{C}^*$ is a good quotient (see Remark \ref{remark:affinegit}). 
Note that the two $\C^*$-bundles $B_{\mp}\to X_{\pm}$ are not locally trivial if the action is non-equalized. Moreover, the $\C^*$-action on $B_{\mp}$ allows to construct two fiber bundles obtained by intuitively adding respectively the zero and the infinity section.

\begin{proposition}	\label{prop:lb_torici}
	Given the $\C^*$-bundles $B_{\mp}\to X_{\pm}$, consider
	$$ E_{\pm}=B_{\pm}\times^{\C^*}\C:=\frac{B_{\pm}\times \C}{\sim}, $$
	where, given $e,e'\in B_{\pm}$, $\lambda,\lambda'\in\C$,
	$$(e,\lambda)\sim(e',\lambda') \iff \exists \hspace{0.1cm} t\in\C^* \text{ s.t. } e'=te \text{ and } \lambda'=t\lambda. $$
	Then $E_{\pm}$ are fiber bundles on $X_{\mp}$ with fibers $\C$, and they inherit a toric structure: indeed the maximal cones of the associated fans can be respectively described as
	\begin{equation*}
		\begin{split}
			\Lambda_+(n+1)&=\{\langle \delta_i,-v \rangle \mid i=d_2+1,\ldots,n+1\}, \\
			\Lambda_-(n+1)&=\{\langle \delta_i,v\rangle\mid i=1,\ldots, d_1\}.
		\end{split}
	\end{equation*}
\end{proposition}

\begin{proof}
	Let us consider the case of $\Lambda_+$, being the other similar. The fan $\Lambda_+$ weakly splits (in the sense of \cite[Exercise 3.3.7]{CLS}) by $\langle -v \rangle$ and $\Delta_+$, hence we conclude it is a fiber bundle with fibers isomorphic to $\C$.
\end{proof}

In analogy with our picture-like description, we can visualize $E_+,E_-$ as:

\begin{center}
	\includegraphics[scale=0.08]{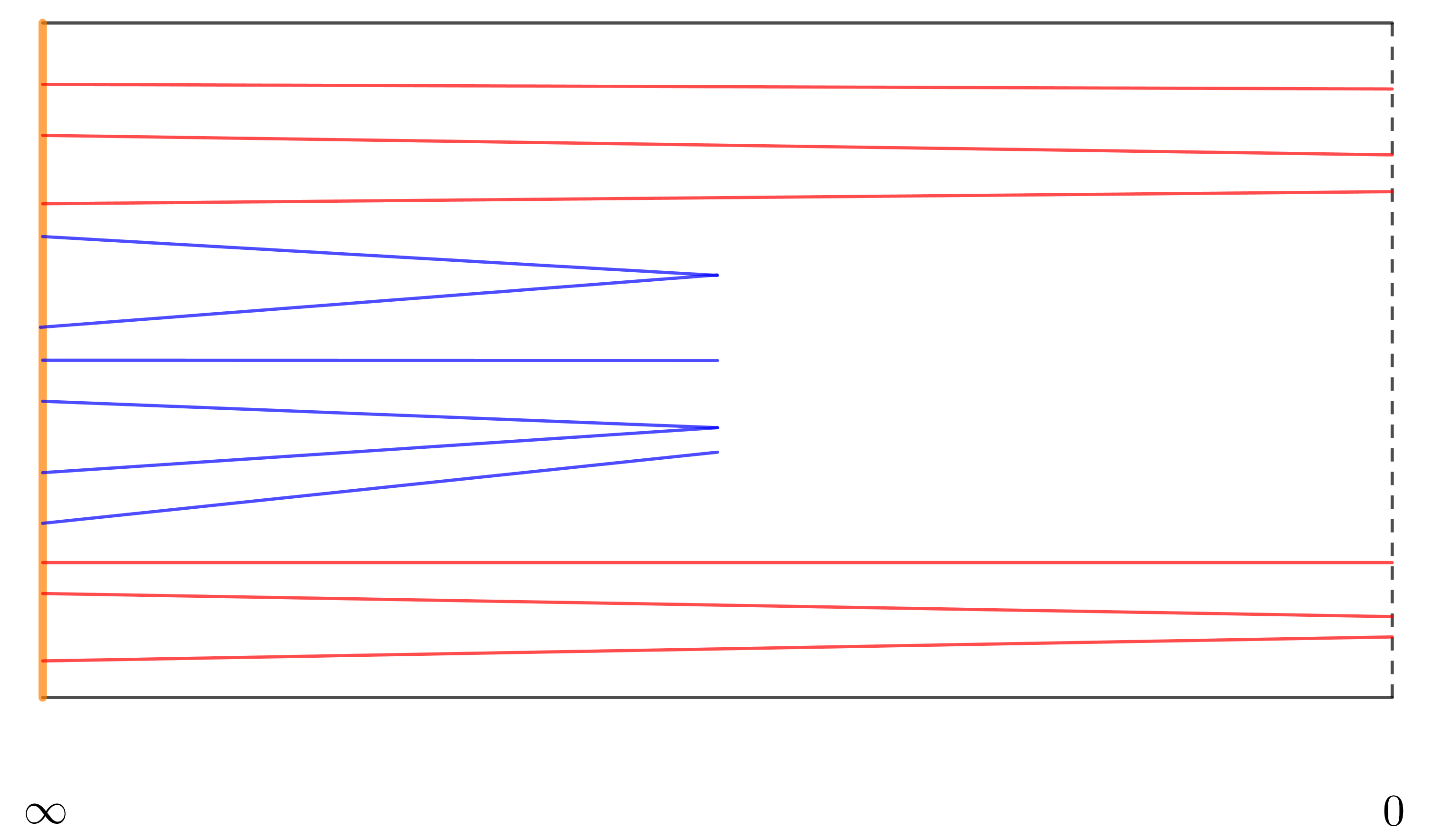}\hspace{0.5cm}\includegraphics[scale=0.08]{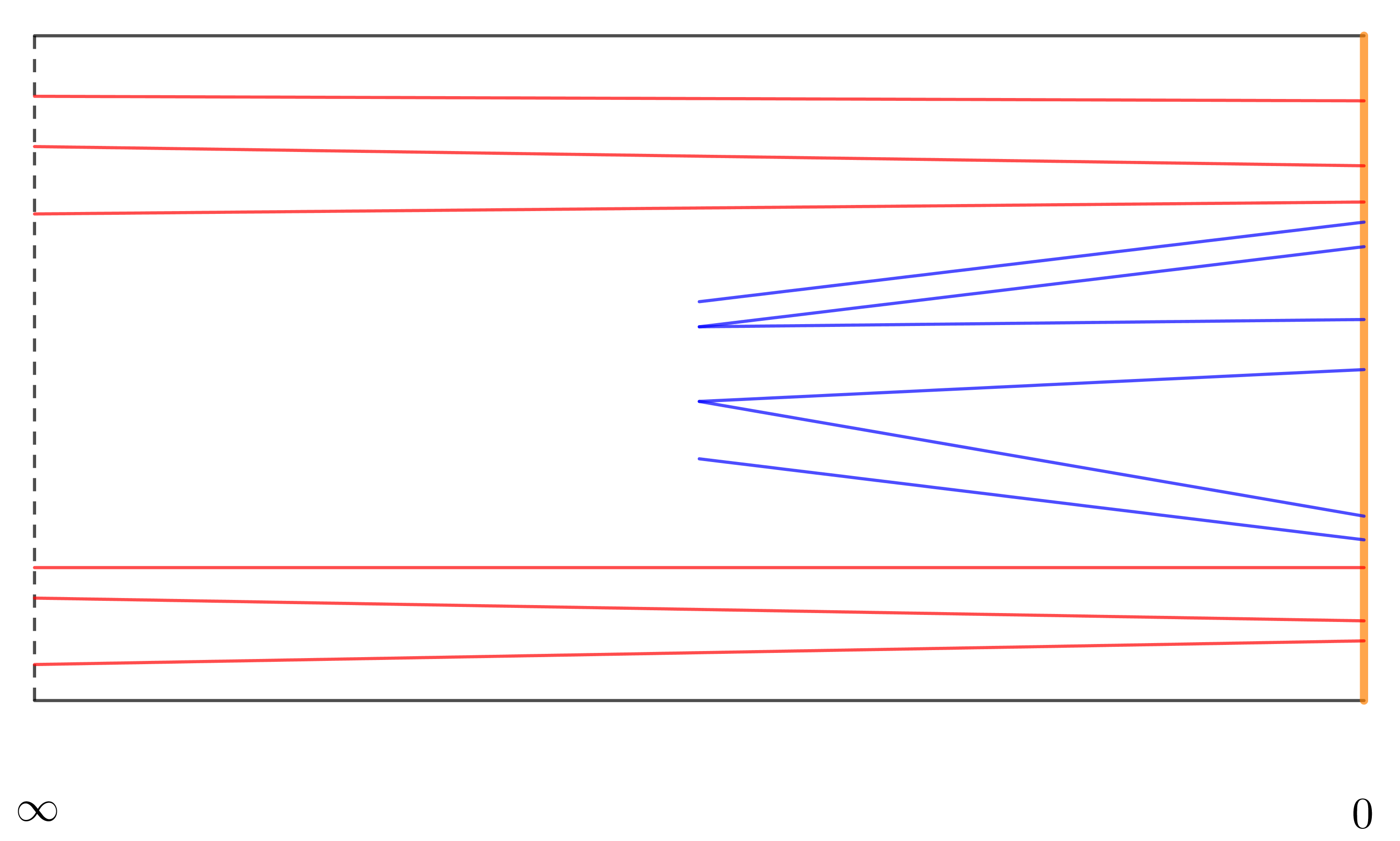}
\end{center}
where the left and right side represent the base spaces $X_-$ and $X_+$, while the horizontal lines correspond to the fibers of the bundles. 

\begin{proposition}\label{toricbordism}
	Let us keep the notations and assumptions of Set-up \ref{setupcobordism}, and let $\Lambda_{\pm}$ be the fans associated to the toric varieties $E_{\pm}$ of Proposition \ref{prop:lb_torici}. Consider a new fan 
	$$\tilde{\Sigma}=\Lambda_+\cup \Sigma(\delta) \cup \Lambda_- $$
	in $N_{\R}$. Then the corresponding toric variety $X_{\tilde{\Sigma}}$ admits a $\C^*$-action associated to the $1$-parameter subgroup $v$. Moreover, the fixed locus of this action is given by $X_-,X_+$ that are respectively the sink and the source, and by an inner component isomorphic to $\C^{d_2-d_1}$.
\end{proposition}

\begin{proof}
	In analogy with \cite[Proposition 5.7]{WORS1}, the action of $\C^*$ associated to $v$ can be extended to the toric variety $X_{\tilde{\Sigma}}$. In order to study the fixed point locus of the action, we consider separately the three toric varieties used in the construction: on one hand $X_-$ and $X_+$ are precisely the zero sections of the fiber bundles $E_{+}$ and $E_-$, thus they will correspond to the sink and the source of the action;  on the other, the inner component on which $\C^*$ acts with zero weights is associated to the affine space of dimension $d_2-d_1$.	
\end{proof}

We thus have constructed a new toric variety by glueing together the following three patches corresponding, respectively, to $E_+$, $X_\delta$ and $E_-$:

\begin{center}
	\includegraphics[scale=0.062]{X-.png}\hspace{0.2cm} \includegraphics[scale=0.062]{cn.png} \hspace{0.2cm} \includegraphics[scale=0.062]{X+.png}
\end{center}

Using all the constructions explained above, we are ready to introduce the notion of toric bordism. 
\begin{definition}\label{definition:toricbordism}
	The equivariant embeddings $X_-\hookrightarrow X_{\tilde{\Sigma}}\hookleftarrow X_+$ are called \emph{toric bordism} associated to the toric flip $X_-\dashrightarrow X_+$.
\end{definition}

\begin{remark} By starting with a $\C^*$-action as in Set-up \ref{setupcobordism} which gives a cobordism among two varieties $X_-$ and $X_+$ (see Remark \ref{rem:cobordism}), one may construct a toric bordism, namely a toric variety $X_{\tilde{\Sigma}}$ admitting a $\C^*$-action having $X_{\pm}$ as extremal fixed components and such that $X_-\dashrightarrow X_+$ is a toric flip either of Atiyah or non-equalized type. If the toric flip $X_-\dashrightarrow X_+$ is of Atiyah type, note that we obtain \cite[Definition 5.8]{WORS1}.
\end{remark}

\subsection{Weighted blow-ups along fixed components} \label{sec:weighted} In this subsection we introduce the notion of weighted blow-up of a $\C^*$-variety $X$ along an extremal fixed component $Y$. 
Our construction is a generalization of \cite[$\S$II.7, p. 163]{Ha}. We refer to \cite{ATW} for a detailed anthology of the use of the weighted blow-up in the literature. We will consider it in the case of $\C^*$-actions; the following construction is probably well known to the experts, but we present it for lack of references. 

Given a smooth projective variety $X$ of arbitrary dimension $n$, assume that we have a $\C^*$-action $\alpha
\colon \C^*\times X\to X$ which is non-equalized at an extremal fixed point component $Y$. As observed in section \ref{sec:prelim}, the subvariety $Y$ is also smooth. 

Set $d:=\dim{Y}$ and denote by $q_{d+1},\dots,q_{n}$ the non-zero weights of the $\C^*$-action on $T_{Y|X}$, with $q_{d+1}\leq \dots \leq q_{n}$. Without loss of generality, we may assume that $Y$ is the source, so that these weights are all positive.  
Let $\cI_{Y}$ be the ideal sheaf of $Y$ on $X$. Starting from $\cI_{Y}$ we will construct a graded finitely generated $\cO_{X}$-algebra. Take $\cU\subset X$ an open subset. Locally analytically in a neighborhood of a general point $z$ of $\cU \cap Y$ we may consider coordinates $y_1,\dots, y_d, x_{d+1},\dots,x_n$ such that $\cU \cap Y$ has equations $x_{d+1}=\dots=x_n=0$, and $y_1,\dots, y_d$ are local coordinates of $\cU\cap Y$ around $z$. Notice that $\C^*$ acts with weights zero on $y_1,\dots, y_d$, and with respectively weights $q_i$, $i=d+1,\dots,n$ on  $x_{d+1},\dots,x_n$. Locally analitically around $z$, a regular function $f$ can be written in a compact way: $$f=f(y_1,\dots, y_d,x_{d+1}, \dots, x_{n})=\sum b_{IJ} x^I y^J,$$
where $x^I=\prod_{i=d+1,\dots,n} x_i^{m_i}$, $y^J=\prod_{j=1,\dots,d} y_j^{s_j}$, and each monomial $x^I y^J$ is set to have degree $\sum_{i=d+1}^{n} q_{i} m_i$. Then we define $$\deg_{\alpha}{f}=\min\{{\deg{x^I y^J}}\}$$ and $$\cS_{\alpha}=\bigoplus_{m\in \ZZ_{\geq 0}} \cI_{Y}^{{\alpha,m}},$$ where $\cI_{Y}^{\alpha,m}(U)=\{f\in \cI_{Y}(U)| \deg_{\alpha}{f}=m\}$, and $\cI_{Y}^{\alpha,0}=\cO_X$. Clearly, $\cS_{\alpha}$ is a graded finitely generated $\cO_{X}$-algebra, so that we may consider the projective variety $X^{\flat}=\Proj{\cS_{\alpha}}$. The natural inclusion $\cO_X\subset \cS_{\alpha}$ provides a birational map $X^{\flat}\to X$. Summarizing the construction we have already introduced, we formulate the notion of a weighted blow-up along an extremal component in the following way. 
\begin{definition} \label{def:wblowup} Let $X$ be a smooth projective variety with a $\C^*$-action $\alpha$ non-equalized at an extremal component $Y\in \cY$. The weighted blow-up $X^\flat$ of $X$ along $Y$ with respect to $\alpha$ is the projective variety constructed above as $X^\flat=\Proj{\cS_{\alpha}}$, together with the birational map $X^{\flat}\to X$. 
\end{definition} 
\begin{remark} \label{rem:wb_extremal} Let us denote by $Y^{\flat}$ the exceptional divisor of the weighted blow-up along $Y$. Running the same arguments of \cite[Theorem II.8.24]{Ha} we obtain that $$Y^{\flat}=\PP_{\alpha}(\cN_{Y|X}^{\vee}):=\frac{\cN_{Y|X}\setminus s_0(Y)}{\sim},$$ where  $s_0(Y)$ is the zero section and $\sim$ denotes the quotient by the action $\alpha$. Using that $\cN^{\pm}(Y)\simeq X^{\pm}(Y)$ (see Theorem \ref{thm:BB_decomposition} (2)) we deduce that if $Y$ is the sink (respectively, the source) then $Y^{\flat}$ parametrizes all the orbits whose sink (respectively, source) lies in $Y$; namely the natural projection $Y^{\flat}\to Y$ is a fiber bundle whose fibers are isomorphic to the weighted projective space $\mathbb{P}(q_{d+1},\dots,q_n)$. 
\end{remark}

\begin{remark} \label{rem:wblow_vs_bl} Notice that in the case in which the action is equalized at the extremal component $Y$, the weighted blow-up of $X$ along $Y$ with respect to the $\C^*$-action coincides with the classical blow-up. 
\end{remark}

We conclude this part by studying the parallelism between the above construction and the well-known toric description of the weighted blow-up. The following Remark follows from the construction given in \cite[Chapter 3.3, pp. 131-132]{CLS}. 

\begin{remark} \label{rem:toric_wb}
Take a lattice of 1-parameter subgroups $N\simeq \Z^n$, and consider a cone $\sigma=\langle e_1,\ldots, e_n\rangle$, where $e_1,\ldots,e_n$ is the canonical basis of $\R^n$, such that the corresponding affine toric variety is $X_{\sigma}=\C^n$. Consider moreover a face $\tau=\langle e_{d+1},\ldots,e_n\rangle$, with $2\leq d < n$, so that by the orbit-cone correspondence one has $\overline{\cO(\tau)}=\{0\}\times \C^d$; take a vector $\omega=(0^d,q_{d+1},\ldots,q_n) \in N$ where the exponent denotes the occurrence of the zero weight and $0<q_{d+1}\leq \dots \leq q_n$. The weighted blow-up of $X_{\sigma}$ along $\overline{\cO(\tau)}$ with weight vector $\omega$ is defined as the toric variety associated to the fan whose maximal dimensional cones are: 
	$$\Sigma^*(\tau)(n)=\{ \langle e_1,\ldots,e_d, \omega, e_{d+1},\dots, \hat{e}_i,\dots, e_n\rangle, \text{\ such that} \ i\in \{d+1,\dots,n\}\},$$
	that is the product fan of the one associated to the weighted blow-up of $\C^{n-d}$ at the origin and of the fan of $\C^d$. It can be proven that the toric morphism $\phi: X_{\Sigma^*(\tau)}\to X_{\sigma}$ is a birational map such that, for any $p\in \overline{\cO(\tau)}$, the preimage is the weighted projective space $\P(q_{d+1},\ldots,q_n)$. Moreover, it can be shown that the map $\beta\colon X^{\flat}\to X$ of Definition \ref{def:wblowup} coincides locally analytically with the toric morphism $\phi: X_{\Sigma^*(\tau)}\to X_{\sigma}$. More concretely, for every point $y\in Y$ there exists an analytic $\C^*$-invariant neighborhood $\cU$ of $y$ in $X$ biholomorphic to $\C^n\simeq X_{\sigma}$, such that $\cU\cap Y=\overline{\cO(\tau)}$, and $\beta^{-1}(\cU)\simeq X_{\Sigma^*(\tau)}$, so that $\beta$ coincides with $\phi$ up to these isomorphisms.  
\end{remark}


\section{Main results}\label{sec:main}

In this section we aim to investigate the birational map described in Lemma \ref{lem:birational_map} arising between exceptional divisors of weighted blow-ups along the sink and source of a variety endowed with a $\C^*$-action. This study has been started in \cite{WORS1} in the case of equalized $\C^*$-actions of small bandwidth. For bandwidth two varieties admitting equalized $\C^*$-actions, it turns out that such a birational map is related to Atiyah flips (see \cite[$\S$7]{WORS1}), while for equalized actions of bandwidth three there is a connection with special Cremona transformations (see \cite[$\S$8]{WORS1}). Here we start to extend the previous investigation to non-equalized actions, dealing with the case of criticality two. Our arguments will involve the notions of toric Atiyah and non-equalized flips (cf. Definition \ref{toricatiyahflip}), and the constructions explained in section \ref{sec:birgeo}, as weighted blow-ups, toric bordism and cobordism. We will keep the notations and assumptions stated as follows. 

\begin{setup} \label{setup_main} Let $X$ be a smooth projective variety with a faithful $\C^*$-action of criticality $r\geq 2$. Let $\beta\colon X^{\flat}\to X$ be the weighted blow-up along the sink $Y_{-}$ and the source $Y_{+}$ with respect to this action. Denote by $Y_{-}^{\flat}$ and $Y_{+}^{\flat}$ the corresponding exceptional divisors, and by $\psi:Y_{-}^{\flat}\dashrightarrow Y_{+}^{\flat}$ the birational map defined in Lemma \ref{lem:birational_map}. We also denote by $Z^\flat_{\pm}$ the strict transform in $X^\flat$ of the varieties $Z_{\pm}$ defined in section \ref{sec:birgeo} (see (\ref{equation_Z})). 
\end{setup}
\begin{remark} Notice that the above assumption on the criticality $r\geq 2$ is not restrictive: indeed, when $r=1$ the birational map defined in Lemma \ref{lem:birational_map} turns out to be an isomorphism, so that for our purpose the first interesting case is $r=2$. 
\end{remark}
\begin{remark} \label{rem:exc_locus} In the situation of Setup \ref{setup_main}, assume that the $\C^*$-action on $X^{\flat}$ is a bordism. Then the birational map $\psi:Y_{-}^{\flat}\dashrightarrow Y_{+}^{\flat}$ is an isomorphism in codimension one. Indeed, by construction, $Z^\flat_{-}$ is the exceptional locus of $\psi$, and by \cite[Corollary 3.7]{WORS1} one has $\codim{(Z^\flat_{-},Y_{-}^\flat)}\geq 2$.
\end{remark}

The following result generalizes \cite[Lemma 3.10]{WORS1} to the case of non-equalized $\C^*$-actions, and it will be the first step towards the study of toric flips associated to $\C^*$-actions.

\begin{lemma} \label{lem:descend_action} In the situation of Setup \ref{setup_main} one has that the $\C^*$-action on $X$ extends to a B-type action on $X^{\flat}$, having $Y_-^{\flat}$, $Y_+^{\flat}$ as sink and source, respectively. 
\end{lemma}

\begin{proof} Since $Y_{\pm}$ are fixed point components, we have natural induced $\C^*$-actions on the corresponding tangent bundles $T_{Y_{\pm}|X}$, thus on $Y_{\pm}^{\flat}$. Hence the $\C^*$-action on $X$ extends to an action on $X^{\flat}$. Moreover, using Remark \ref{rem:wb_extremal} we observe that the $\C^*$-action is trivial on the fibers of the weighted projective bundles $Y_{\pm}^{\flat}\to Y$, therefore we conclude that $Y_{-}^{\flat}$ and $Y_{+}^{\flat}$ are respectively the sink and the source of the $\C^*$-action on $X^{\flat}$. 
\end{proof}

\begin{theorem} \label{thm:main} Let $X$ be a smooth projective variety with a $\C^*$-action of criticality $r=2$ as in Setup \ref{setup_main}. Assume that $X^{\flat}$ is a bordism. Then $\psi$ is locally a toric Atiyah flip if and only if the $\C^*$-action on $X^\flat$ is equalized at every inner component. 
\end{theorem} 
\begin{proof}
We recall by Lemma \ref{lem:descend_action} that $Y_-^{\flat}$, $Y_+^{\flat}$ are respectively the sink and source of the $\C^*$-action on $X^\flat$, and as observed in Remark \ref{rem:exc_locus} the map $\psi$ is an isomorphism in codimension one, having $Z_-^{\flat}$ as exceptional locus. Being $r=2$ by assumption, we deduce that the fixed locus of this action is given by $Y_-^{\flat} \sqcup Y_1 \sqcup Y_+^{\flat}$, where $Y_1$ is the union of the irreducible inner components, which are all associated to the same weight. Assume that $\psi$ is locally a toric Atiyah flip. Suppose by contradiction that there exists an irreducible component $Y^{\prime}$ of $Y_1$ on which the action is non-equalized. Choose a point $p\in Y^{\prime}$. Using \cite[Theorem 2.5]{BB} there exists an analytic neighborhood $\cU\subset X$ of $p$, which is $\C^*$-invariant and biholomorphic to $\cN_{Y^{\prime}\cap \cU|X^{\flat}}\simeq \C^{\dim{X^{\flat}}}$. Let us consider the following two geometric quotients of $\cU$:
$$U_{-}=\{y\in Y_-^{\flat}|\,\, y=\lim_{t\to \infty} tx, \,\, x\in \cU\} \ \ \text{and} \ \ U_{+}=\{y\in Y_+^{\flat}|\,\, y=\lim_{t\to 0} tx, \,\, x\in \cU\}.$$
Notice that locally in $\cU$ the $\C^*$-action is as in Set-up \ref{setupcobordism}. Since the weights of the $\C^*$-action on $\cU$ corresponds to the weights of the $\C^*$-action on $\cN_{Y^{\prime}\cap \cU|X^{\flat}}$ and the action on $Y^{\prime}$ is non-equalized by assumption, we deduce that $\psi_{\mid U_{-}}\colon U_-\dashrightarrow U_{+}$ is a toric non-equalized flip, hence a contradiction. 

For the converse, we take a point $z\in Z_-^{\flat}$ and we prove that there exists an open subset of $U_{-}(z)$ of $z$ contained in $Y_-^{\flat}$ such that $\psi_{\mid U_{-}(z)}$ is a toric Atiyah flip. Since $Z_-^{\flat}\subset Y_-^{\flat}$, and by Theorem \ref{thm:BB_decomposition} one has $X^{-}(Y_-^{\flat})\simeq \cN_{Y_-^{\flat}|X^{\flat}}$, it follows that there exists a unique orbit $C$ having sink in $z$. The source $z^{\prime}$ of $C$ sits in one among the fixed point component of $Y_1$. Denote by $\bar{Y}$ such a component. Using again \cite[Theorem 2.5]{BB}, we may find an analytic neighborhood $\cU(z^{\prime})$ of $z^{\prime}$ which is $\C^*$-invariant and biholomorphic to $\cN_{\bar{Y}\cap \cU(z^{\prime})|X^{\flat}}\simeq \C^{\dim{X^{\flat}}}$, and we take two geometric quotients $U_{\pm}(z)$ of $\cU(z^{\prime})$ defined as above. By assumption, the $\C^*$-action is equalized at $\bar{Y}$, then it follows that $\psi_{\mid U_{-}(z)}\colon U_{-}(z) \dashrightarrow U_{+}(z)$ is a toric Atiyah flip, and we conclude. 
\end{proof}
\begin{remark} Let $X$ be as in Theorem \ref{thm:main}, and take a point $p$ in an inner component $Y^{\prime}$. Let $\cU\subset X$ be a $\C^*$-invariant analytic neighborhood of $p$ as in the above proof and define $U_{\pm}$ from $\cU$ in the same way. Using that $\C^*$ acts on $\cU$ as in Set-up \ref{setupcobordism}, we may run the construction developed in section \ref{sec:toric_bordism} to obtain that there exists a toric variety $X_{\tilde{\Sigma}}$ such that $U_-\hookrightarrow X_{\tilde{\Sigma}}\hookleftarrow U_+$ is a toric bordism. If the action is non-equalized at $Y^{\prime}$ then $U_-\dashrightarrow U_+$ is a toric non-equalized flip, therefore our construction is a generalization of the one introduced in \cite[$\S$5.3]{WORS1} where the toric bordism has been studied in the context of toric Atiyah flip. This phenomenon can be described by means of the following picture:
\begin{center}
	\includegraphics[scale=0.13]{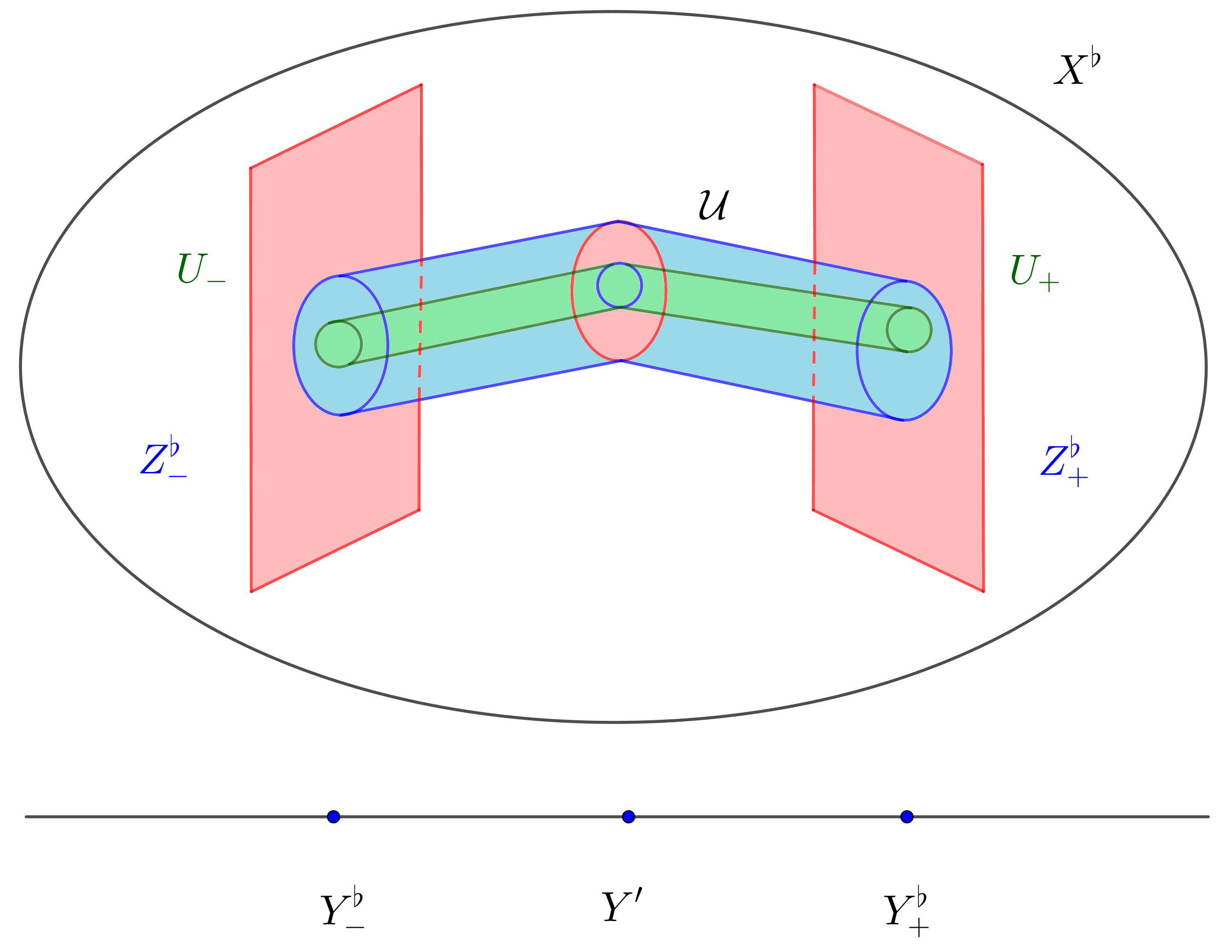}
\end{center}
\end{remark}

The following result is an easy consequence of Theorem \ref{thm:main} that will be used in our applications of section \ref{sec:examples}.
\begin{corollary} \label{cor:main} Let $X$ be a smooth projective variety with a $\C^*$-action of criticality $r=2$ as in Setup \ref{setup_main}. Assume that $\rho_X=1$ and that $\dim{Y_{\pm}}>0$. Then $\psi$ is locally a toric Atiyah flip if and only if the $\C^*$-action on $X^\flat$ is equalized at every inner component. 
\end{corollary} 

\begin{proof} If $\rho_X=1$ and $\dim{Y_{\pm}}>0$ then \cite[Lemma 2.6 (1)]{WORS1} implies that $\nu^\pm(Y)\geq 2$ for every inner component $Y$. Since the weighted blow-up along the sink and the source does not change the values $\nu^\pm(Y)$, it follows that $X^{\flat}$ is a bordism. Then the statement follows by Theorem \ref{thm:main}.
\end{proof}

We may easily extends the above results to arbitrary criticality, by requiring that every inner component $Y\in \cY$ has the following property: there exist orbits having sink at $Y_{-}$ and source at $Y$ and there exists orbits having sink at $Y$ and source at $Y_{+}$; moreover do not exist closures of orbits joining $Y$ with other fixed components different from $Y_{\pm}$. We will formulate this property, by using the following partial order among the fixed components in $\cY$ (see \cite{BBS1}):
\begin{equation} \label{partial_order}
Y\prec Y^{\prime}\Leftrightarrow \exists \ x\in X:\
\lim_{t\rightarrow 0} t x\in Y{\rm \ \ and \ \
}\lim_{t\rightarrow \infty} t x\in Y^{\prime} .
\end{equation}

By requiring such a property among the fixed point components, one may run the same proof of Theorem \ref{thm:main} to get the following:

\begin{corollary} \label{cor:maincrit} Let $X$ be a smooth projective variety as in Setup \ref{setup_main}, such that $X^{\flat}$ is a bordism. Moreover, assume that for every inner component $Y\in \cY$ one has that $Y_-$ is the unique fixed component such that $Y\prec Y_-$, and $Y_+$ is the unique fixed component such that $Y_+\prec Y$. Then $\psi$ is locally a toric Atiyah flip if and only if the $\C^*$-action on $X^\flat$ is equalized at every inner component. 
\end{corollary}
\section{Examples}\label{sec:examples}

\subsection{Non-equalized actions admitting an Atiyah flip} \label{Atiyah_example}

We now study in detail examples of $\C^*$-action of criticality $2$ on a smooth quadric, non-equalized at the extremal components but equalized at the inner component, such that the birational map $\psi$ among the exceptional divisors of the sink and the source as defined in Lemma \ref{lem:birational_map}, can be locally described as a toric Atiyah flip.

Let us consider the quadric hypersurface $Q^{2n-1}=Z(x_0x_{n+1}+\ldots+x_{n-1}x_{2n}+x_n^2)\subset \P^{2n}$; and a family of $\C^*$-actions, which we will denote by $\alpha_k$, defined as follows, for $k=1,\ldots,n$:
$$        \alpha_k\colon H_k\times \P^{2n}\to \P^{2n} $$ 
$$        (t,p)\to [tp_0:\ldots:tp_{k-1}:p_{k}:\ldots:p_n:t^{-1}p_{n+1}:\ldots:t^{-1}p_{n+k}:p_{n+k+1}:\ldots:p_{2n}].$$
Clearly $Q^{2n-1}$ is $H_k$-invariant, and we may describe some properties of such action as we will see in the following discussion. We abuse notation by setting $Q^{-1}\coloneqq\emptyset$.

For the sake of notation, we will denote a point $p\in \P^{2n}$ by
	$$p=[p_+:p_0:p_-],$$
where $p_+,p_0$ and $p_-$ represent the coordinates on which $H_k$ acts with respectively positive, zero and negative weights. We also set $\P^{k-1}_-:=\{p\in \P^{2n}\mid p=[p_+:0:0]\}$ and $\P^{k-1}_+:=\{p\in \P^{2n}\mid p=[0:0:p_-]\}$.

\begin{lemma} \label{lem:prelim_ex1}
The fixed locus of the $H_k$-action on $Q^{2n-1}$ is $(Q^{2n-1})^{H_k}=\P^{k-1}_-\sqcup Q^{2n-2k-1}\sqcup \P^{k-1}_+$ where $\P^{k-1}_{-}$, $\P^{k-1}_+$ correspond respectively to the sink and the source.
\end{lemma}

\begin{proof}
The fixed locus $(Q^{2n-1})^{H_k}$ is obtained by intersecting the quadric with $$(\P^{2n})^{H_k}=\P^{k-1}_-\sqcup \P^{2n-2k}\sqcup \P^{k-1}_+,$$ and using that $\P^{k-1}\cap Q^{2n-1}=\P^{k-1}$, and that $\P^{2n-2k}\cap Q^{2n-1}=Q^{2n-2k-1}$. We are left to observe that $\P^{k-1}_{\pm}$ are the sink and the source of the $H_k$-action on $\P^{2n}$. Indeed, given a general point $p\in \P^{2n}$, we easily obtain that 
	$$\lim_{t\to \infty} tp=\lim_{t\to \infty} [tp_+:p_0:t^{-1}p_+]=\lim_{t\to \infty} [p_+:t^{-1}p_0:t^{-2}p_+]=[p_+:0:0],$$
	which belongs to $\P^{k-1}_-$ by construction, and similarly $\lim_{t\to 0} tp \in \P^{k-1}_+$, hence the claim. 
\end{proof} 

According with the notation of the previous sections, we keep denoting by $Y_{\pm}$ the sink and the source of the action.

\begin{proposition}\label{bnbandwidth}
	The action of $H_k$ on $Q^{2n-1}$ has criticality $2$.
\end{proposition}

\begin{proof}
Since $\Pic{Q^{2n-1}}\simeq \Z$, the criticality of the $H_k$ on $Q^{2n-1}$ does not depend on the chosen ample line bundle on $Q^{2n-1}$. For simplicity, consider $L=\cO_{Q^{2n-1}}(1)$. We obtain that $\mu_{L}(Y_{\pm})=\pm 1$, and $\mu_{L}(Q^{2n-2k-1})=0$, hence the statement. 
\end{proof}

\begin{remark} \label{rem:BW2}
Keep the notation of Proposition \ref{bnbandwidth} and of its proof. By definition, the bandwidth of the $H_k$-action on the polarized pair $(Q^{2n-1},L)$ is $\mu_{L}(Y_+)-\mu_{L}(Y_-)=2$.
\end{remark}

\begin{lemma} \label{lem:equalized}
The $H_k$-action on $Q^{2n-1}$ is equalized if and only if $k=1$.
\end{lemma}

\begin{proof}
If $k=1$, the sink and the source are two isolated points, and using Remark \ref{rem:BW2} and \cite[Theorem 4.1]{RW} we deduce that the action is equalized. 

Assume that $k\neq 1$, and take $L=\cO_{Q^{2n-1}}(1)$. Consider a point $p=[p_+:0:p_-]\in Q^{2n-1}$, and denote by $C$ the closure of the orbit $H_k\cdot p$. Then $C$ is the line of the form $[tp_+:0:t^{-1}p_-]$, for $t\in\C^*$, and applying Lemma \ref{lem:AMvsFM} we get $2=\delta(\tilde{p})\deg L$, where $\tilde{p}$ is the source of $C$. Since $\deg L=1$, one has $\delta(\tilde{p})=2$, thus the action is non-equalized.
\end{proof}
\begin{remark} Under the above notation, assume that $k=1$. Then by Lemmas \ref{lem:equalized} and \ref{lem:prelim_ex1} we know that the action is equalized and the sink and source are two isolated fixed points. If we consider the blow-up of $Q^{2n-1}$ along the extremal fixed components (cf. Remark \ref{rem:wblow_vs_bl}), then the birational map $\psi$ of Lemma \ref{lem:birational_map} corresponds to the identity $\P^{2n-2}\to \P^{2n-2}$; thus we will be interested in the cases $k\neq 1$. 
\end{remark}

\begin{proposition} \label{prop:normal}
	The normal bundle of $Y_{\pm}$ can be decomposed as
	$$\cN_{Y_{\pm}\mid Q^{2n-1}}=\Omega_{Y_{\pm}}(2)\oplus \cO_{Y_{\pm}}(1)^{\oplus 2n-2k+1}.$$
\end{proposition}

\begin{proof} We recall by Lemma \ref{lem:prelim_ex1} that $Y_-\simeq Y_+\simeq \P^{k-1}$; set $Y:=Y_{\pm}$. Notice that $\mathcal{N}_{Y\mid \P^{2n}}=\mathcal{O}_{Y}(1)^{\oplus 2n-k+1}$ and $\mathcal{N}_{Q^{2n-1}\mid \mathbb{P}^{2n}}|_{Y}=\mathcal{O}_{Y}(2)$. Then we have the following exact sequence $$0 \to \mathcal{N}_{Y\mid Q^{2n-1}} \to \mathcal{O}_{Y}(1)^{\oplus 2n-k+1}\to \mathcal{O}_{Y}(2)\to 0.$$ Being $\Hom{(\mathcal{O}_{Y}(1), \mathcal{O}_{Y}(2))}\simeq H^0(Y,\mathcal{O}_{Y}(1))$, the surjective map $\mathcal{O}_{Y}(1)^{\oplus 2n-k+1}\to \mathcal{O}_{Y}(2)$ is given by a set of sections $f_0,\dots, f_{2n-k}$ which generate $H^0(Y,\mathcal{O}_{Y}(1))$. For simplicity, assume that $f_0,\dots,f_k$ give a basis of $H^0(Y,\mathcal{O}_{Y}(1))$.  Let us consider the exact sequence:
$$ 0 \to \cO_{Y}(1)^{\oplus k} \to \mathcal{O}_{Y}(1)^{\oplus 2n-k+1} \to \mathcal{O}_{Y}(1)^{\oplus 2n-2k+1} \to 0.$$ 	
Using that there exist surjective maps $\cO(1)^{\oplus k} \to \mathcal{O}_{Y}(2)$, $\mathcal{O}_{Y}(1)^{\oplus 2n-k+1} \to \mathcal{O}_{Y}(2)$, we get the following commutative diagram:
\begin{center}
		\begin{tikzcd}
			\Omega_{Y}(2) \arrow[r] \arrow[d]  & \mathcal{O}_{Y}(1)^{\oplus k} \arrow[r] \arrow[d]                                             & \cO_{Y}(2) \arrow[d, no head, Rightarrow, no head]         &  &   \\
			\mathcal{N}_{Y\mid Q^{2n-1}} \arrow[r] \arrow[d]                      & \mathcal{O}_{Y}(1)^{\oplus 2n-k+1} \arrow[d] \arrow[r]                               & \cO_{Y}(2) &  &  \\
			\mathcal{O}_{Y}(1)^{\oplus 2n-2k+1} \arrow[r, no head, Rightarrow, no head]  & \mathcal{O}_{Y}(1)^{\oplus 2n-2k+1}   &  &  
		\end{tikzcd}
	\end{center}
	From the above diagram we obtain the exact sequence: 
	$$0 \to \Omega_{Y}(2) \to \mathcal{N}_{Y\mid  Q^{2n-1}} \to \mathcal{O}_{Y}(1)^{\oplus 2n-2k+1}\to 0,$$
	Being $ \Ext^1(\mathcal{O}_{Y}(1)^{\oplus 2n-2k+1},\Omega_{Y}(2)) = \bigoplus^{2n-2k+1}\text{H}^1(Y,\Omega_{Y}(1)) = 0,$
	the claim follows.
\end{proof}


\begin{proposition} \label{prop:ex1}
Consider the weighted blow-up of $Q^{2n-1}$ along the sink $Y_-$ and source $Y_+$ with respect to the $H_k$-action, with $k\neq 1$. Denote by $Y_{\pm}^{\flat}$ the corresponding exceptional divisors. Then $Y_{\pm}^{\flat}\simeq \P_{\alpha_k}({(\Omega_{Y_{\pm}}(2)\oplus \cO_{Y_{\pm}}(1)^{\oplus 2n-2k+1})}^\vee)$, and the birational map introduced in Lemma \ref{lem:birational_map}
	$$\P_{\alpha_k}({(\Omega_{Y_{-}}(2)\oplus \cO_{Y_{-}}(1)^{\oplus 2n-2k+1})}^\vee)\stackrel{\psi}{\dashrightarrow} \P_{\alpha_k}({(\Omega_{Y_{+}}(2)\oplus \cO_{Y_{+}}(1)^{\oplus 2n-2k+1}})^\vee),$$
	which associates to every point $p\in Y_-^\flat$ the source of the unique orbit having $p$ as sink, is locally a toric Atiyah flip.	
\end{proposition}

\begin{proof} By Proposition \ref{prop:normal} one has $Y_{\pm}^{\flat}\simeq \P_{\alpha_k}({(\Omega_{Y_{\pm}}(2)\oplus \cO_{Y_{\pm}}(1)^{\oplus 2n-2k+1})}^\vee)$.  Proposition \ref{bnbandwidth} tells us that the $H_k$-action has criticality two, and 
applying Lemma \ref{lem:AMvsFM} we deduce that the $H_k$-action is equalized at the inner component. Therefore, Corollary \ref{cor:main} gives the statement. 
\end{proof}

\subsection{Non-equalized action admitting a non-equalized flip}\label{example_noneq}

We present an example of a non-equalized $\C^*$-action of criticality $2$, such that the birational map among the exceptional divisors of the weighted blow-up at the sink and the source (see Lemma \ref{lem:birational_map}) is locally described by a toric non-equalized flip.

Keeping the notation introduced at the beginning of subsection \ref{Atiyah_example}, we start by considering the $H_n$-action on $Q^{2n-1}$, that by Lemma \ref{lem:prelim_ex1} has fixed locus $(Q^{2n-1})^{H_n}=\P^{n-1}_-\sqcup \P^{n-1}_+$, where the sign $\pm$ denote respectively the sink and the source of the action. Therefore, we may consider the induced action of $H_n$ on every Grassmannian $G(\P^{i-1},Q^{2n-1})$, for $i\in\{1,\ldots,n\}$, as the restriction of the induced $H_n$-action on $\P(\bigwedge^{i+1} V)\simeq \P^{N}$, where $\dim{V}=2n+1$ and $N=\binom{2n+1}{i+1}-1$, via the Pl\"ucker embedding. In the following discussion, we set $X:=G(\P^1,Q^{2n-1})$.

\begin{lemma} \label{lem:fixedlocus_ex2}
	Consider the induced $H_n$-action on $X$. Then the fixed locus decomposes as
	$$X^{H_n}=G(\P^1,\P^{n-1}_-)\sqcup \P(T_{\P^{n-1}}) \sqcup G(\P^1,\P^{n-1}_+).$$
\end{lemma}

\begin{proof} The fixed components of $X$ under the $H_n$-action are precisely the $H_n$-invariant lines with respect to the $H_n$-action on $Q^{2n-1}$; therefore we find the varieties $G(\P^1,\P^{n-1}_{\pm})$. We are left to study the $H_n$-invariant lines from $\P^{n-1}_-$ to $\P^{n-1}_+$. Note that the intersection between $Q^{2n-1}$ and the subspace generated by $\P^{n-1}_{\pm}$ is a quadric $Q^{2n-2}$. Consider therefore a point $p_{-}\in \P^{n-1}_-$, and the set:
	$$H(p_{-}):=\{p_+\in \P^{n-1}_+\mid \overline{p_{-} p_{+}}\in Q^{2n-2}\}.$$
	It is easy to see that $H(p_{-})$ is an hyperplane in $\P^{n-1}_+$, and that the map
	\begin{equation*}
		\begin{split}
			D: \P^{n-1}_- &\to (\P^{n-1}_-)^\vee\\
			p_{-}&\mapsto H(p_{-})
		\end{split}
	\end{equation*}
	is an isomorphism, therefore the $H_n$-invariant lines from $\P^{n-1}_-$ to $\P^{n-1}_+$ are given by the choice of a point and the associated hyperplane, that is the variety $\P(T_{\P^{n-1}})$.
\end{proof}

According with the notation of the previous sections, we will denote by $Y_{\pm}=G(\P^1,\P^{n-1}_{\pm})$; the reason is clear thanks to the following:	

\begin{lemma}\label{lem:prelim_ex2}
	The $H_n$-action on $X=G(\P^1,Q^{2n-1})$ has criticality $2$. Moreover $Y_-$ and $Y_+$ are respectively the sink and the source of the action.
\end{lemma}

\begin{proof}
	We consider the $H_n$-action on the pair $(G(\P^1,Q^{2n-1}),L)$, where $L=\cO_{G(\P^1,Q^{2n-1})}(1)$. As we have observed above, the action of $H_n$ on $G(\P^1,Q^{2n-1})$ is the restriction of the induced $H_n$-action on $\P(\bigwedge^2 V)$, where $\dim V=2n+1$; we will denote by $\{e_0,\ldots,e_{2n}\}$ a basis of $V$.
	
We first notice that $\mu_{L}(Y_-)=-2$; indeed the induced $H_n$-action via the Pl\"ucker embedding will act on $e_i\wedge e_j$ with weight $2$, for $0\leq i < j \leq n-1$. Similarly we get $\mu_{L}(Y_+)=2$, and $\mu_{L}(\P(T_{\P^{n-1}}))=0$. Therefore by definition we obtain that the criticality of the $H_k$-action is $2$. Moreover, $Y_{-}$ and $Y_{+}$ are the sink and the source of the action, because their associated weights $\mu_{L}(Y_{-})$ and $\mu_{L}(Y_{+})$ are respectively the minimum and the maximum values of $\mu_{L}$.
\end{proof}

\begin{remark} Under the notation of Lemma \ref{lem:prelim_ex2} and of its proof, the bandwidth of the $H_n$-action on $(X,L)$ is $\mu_{L}(Y_+)-\mu_{L}(Y_-)=4$.
\end{remark}

Let us now show that the $H_n$-action on $X$ is non-equalized at the inner component $\P(T_{\P^{n-1}})$; in the following notation the exponent denotes the occurrence of the corresponding weight.

\begin{lemma} \label{lem:prelim2}
	The weights of the $H_n$-action on the normal bundle $\cN_{\P(T_{\P^{n-1}})\mid X}$ are $(\pm 1,\pm 2^{n-2})$.
\end{lemma}

\begin{proof}
	Set $Y_{1}:=\P(T_{\P^{n-1}})$ and as in the proof of Lemma \ref{lem:fixedlocus_ex2}, we denote by $H(p)$ the hyperplane in $\P^{n-1}_+$ corresponding to a point $p\in \P^{n-1}_-$. Let us compute the weights of the $H_n$-action on $\cN^{+}_{Y_1|X}$. To this end, take a point $s\in Y_{1}$ and let us denote by $p_-,p_+$ its intersection with $\P^{n-1}_-,\P^{n-1}_+$ respectively. Then the family of pencils of lines in $Q^{2n-1}$ containing $s$ and a line $r\in G(\P^1,\P^{n-1}_{+})$ is parametrized by the lines passing by $p_+$ contained in the hyperplane $H(p_-)\subset \P^{n-1}_+$. Since $H(p_-)\simeq \P^{n-2}$ we deduce that such a family is parametrized by a $\P^{n-3}$ in $Y_+$. This implies that we have $(n-2)$-independent directions from $s$ that correspond to orbits lying in $X^{-}(Y_{1})$. Denote by $\Gamma$ the closure of one among these orbits. Noticing that $\Gamma$ has sink at $Y_-$, and using Lemmas \ref{lem:prelim_ex2} and \ref{lem:AMvsFM}, we compute that the weight of the $H_n$-action on the tangent bundle of $\Gamma$ at $p_-$ is 2. 
	
Moreover, since it is readily seen from the computation of the rank of $\cN_{Y_1\mid X}$ that $\nu^+(Y_1)=n-1$,
	by Lemma \ref{lem:AMvsFM} one may find a $H_n$-invariant non-singular conic linking $s$ with $Y_+$, and now the weight of the $H_n$-action on the tangent bundle of such a conic at $p_-$ is 1. Then, applying Theorem \ref{thm:BB_decomposition} we deduce that the positive weights of $H_n$ on $s$ are $(1,2^{n-2})$. Running a symmetric argument replacing $G(\P^1,\P^{n-1}_{+})$ with $G(\P^1,\P^{n-1}_{-})$ we conclude that the weights of the $H_n$-action on $\cN^{-}_{Y_1|X}$ are $(-1,-2^{n-2})$, hence the statement.
\end{proof}

In order to construct explicitly the birational map among the exceptional divisors associated to the sink and the source, we compute the normal bundle of $G(\P^1, \P^{n-1})$ in $X=G(\P^1,Q^{2n-1})$.

\begin{proposition} \label{prop:normal_ex2}
	Let $\cQ$ and $\cS$ be respectively the universal quotient bundle and the universal subbundle of $Y:=G(\P^1,\P^{n-1})$. Then $\cN_{Y|X}=\cN^{\prime}\oplus \cQ$ where $\cN^{\prime}$ is the Atiyah extension of $\cO_{Y}(1)$ by $\Omega_{Y}(1)$. 
\end{proposition}

\begin{proof} It is well known that the tangent bundle $T_{Y}$ of $Y$ is isomorphic to $\cQ\otimes \cS$. Consider the universal sequence on $G(\P^1,\P^{2n})$ and denote by $\bar{S}$ the restriction on $Y$ of the universal subbundle of $G(\P^1,\P^{2n})$. Notice that $\cQ$ is the restriction on $Y$ of the universal quotient bundle of $X$. The quadric form defining the quadric $Q^{2n-1}\subset \P^{2n}$ gives a surjective morphism $\bar{S}\to \cQ$ of bundles over $X$; we denote by $\cU$ its kernel. Using \cite[Proposition 5.1]{LM} we get the unsplit exact sequence:
$$0\longrightarrow \cQ\otimes \cU\longrightarrow T_{X}\longrightarrow \cO_{X}(1)\longrightarrow 0,$$
and considering the restriction on $Y$ of the sequence above, we obtain the following commutative diagram, where we continue to denote by $\cU$ its restriction to $Y$: 
\begin{center}
		\begin{tikzcd}
			\cQ\otimes \cS \arrow[r, "\simeq"] \arrow[d]  & T_{Y} \arrow[d] & \\
			\cQ\otimes \cU \arrow[r] \arrow[d]                      & T_{X\mid Y} \arrow[d] \arrow[r]                               & \cO_{Y}(1) \arrow[d, no head, Rightarrow, no head]   &  &  \\
			\cQ\otimes \cU/\cS \arrow[r]  & \mathcal{N}_{Y|X} \arrow[r]  & \cO_{Y}(1) &  
		\end{tikzcd}
	\end{center}
and one may check that $\cU/\cS=\cS^{\vee}\oplus \cO_{Y}$. Therefore $\cQ\otimes \cU\simeq (\cQ\otimes\cS^\vee)\oplus \cQ$, and from the above diagram we obtain the following exact sequence: 
$$0\longrightarrow  (\cQ\otimes\cS^\vee)\oplus \cQ \longrightarrow \mathcal{N}_{Y|X}\longrightarrow \cO_{Y}(1)\longrightarrow 0.$$
Being $\Ext^1(\cO_Y(1),\cQ)=H^1(\cQ(-1))=0$ (cf. \cite[Proposition 2.3]{AT}) one has:\\ $\Ext^1(\cO_Y(1),(\cQ\otimes\cS^\vee)\oplus \cQ)=\Ext^1(\cO_Y(1),\cQ\otimes \cS^{\vee})=\Ext^1(\cS,\cQ(-1))=\Ext^1(\cS,\cQ^{\vee})\\ \simeq H^1(Y, \cS^{\vee}\otimes \cQ^{\vee})\simeq H^1(Y,\Omega_{Y})\simeq \C$, where we have used that, since $\cQ$ has rank two, it holds $Q\simeq Q^\vee(1)$; therefore we have $\cN_{Y|X}\simeq \cN^{\prime}\oplus \cQ$, where $\cN^{\prime}$ fits in the following exact sequence 
$$0\longrightarrow\cQ^{\vee}(1)\otimes \cS^{\vee}\longrightarrow \cN^{\prime}\longrightarrow \cO_Y(1)\longrightarrow 0,$$ and being $\cQ^{\vee}(1)\otimes \cS^{\vee}\simeq \Omega_{Y}(1)$, we get the statement.

\end{proof}

\begin{proposition} \label{prop:ex2}
Let us take the $H_n$-action on $X=G(\P^1,Q^{2n-1})$ as above. Consider the weighted blow-up of $X$ along the sink $Y_-$ and source $Y_+$ with respect to this action, and denote by $Y_{\pm}^{\flat}$ the corresponding exceptional divisors. Then $Y_{\pm}^{\flat}\simeq \P_{\alpha_n}({\cN^{\prime}}^{\vee}\oplus \cQ^{\vee})$, with $\cN^{\prime}$ and $\cQ$ as in Proposition \ref{prop:normal_ex2}, and the birational map introduced in Lemma \ref{lem:birational_map}
	$$\P_{\alpha_n}(\cN'^\vee\oplus \cQ^\vee)\stackrel{\psi}{\dashrightarrow}\P_{\alpha_n}(\cN'^\vee\oplus\cQ^\vee),$$
	which associates to every point $p\in Y_-^\flat$ the source of the unique orbit having $p$ as sink, is locally a toric non-equalized flip.	
\end{proposition}

\begin{proof} By Proposition \ref{prop:normal_ex2} one has $Y_{\pm}^{\flat}\simeq  \P_{\alpha_n}({\cN^{\prime}}^{\vee}\oplus \cQ^{\vee})$. Using Lemma \ref{lem:prelim_ex2} we know that the $H_n$-action has criticality two. Since the $H_n$-action is non-equalized at the inner component (cf. Lemma \ref{lem:prelim2}), applying Corollary \ref{cor:main} we obtain the statement. 
\end{proof}

\bibliographystyle{plain}
\bibliography{bibliomin}
\end{document}